\documentclass[sts,
	reqno,
]{imsart}

\usepackage{amsthm, amssymb, amsmath}
\usepackage{mathtools}
\usepackage{euscript}
\usepackage{color}
\usepackage{tikz}
\usetikzlibrary{arrows}
\usepackage[utf8]{inputenc}
\usepackage[shortlabels]{enumitem}
\usepackage{url}
\usepackage{etex}
\usepackage{alltt}
\usepackage{memhfixc}
\usepackage{bbold}%
\usepackage{cases}
\usepackage[noadjust]{cite}
\definecolor{labelkey}{rgb}{0.0, 0.8, 0.3}
\usepackage[top=2.5cm,bottom=2.5cm,left=2.5cm,right=2.5cm,includeheadfoot,asymmetric]{geometry}
\usepackage{algorithm}
\usepackage{xfrac}
\usepackage{wrapfig}

\usepackage{custom}

\numberwithin{equation}{section}
\usepackage{hyperref}
\hypersetup{
  colorlinks = true,
  urlcolor = blue, 
  linkcolor = blue,
  citecolor = red}
  
\declaretheorem[name=Corollary]{cor}
\declaretheorem[name=Definition, style=definition]{defn}

\declaretheorem[name=Lemma]{lem}
\declaretheorem[name=Proposition]{prop}

\declaretheorem[name=Remark, style=remark]{rmk}
\declaretheorem[name=Theorem]{thm}

\renewcommand{\abstractname}{Abstract}

\begin{document}

\begin{frontmatter}

	\title{Exponential ergodicity of mirror-Langevin diffusions}
	\runtitle{Mirror-Langevin diffusions}
	\author{Sinho Chewi \hfill schewi@mit.edu \\
	Thibaut Le Gouic \hfill tlegouic@mit.edu \\
	Chen Lu \hfill chenl819@mit.edu \\
	Tyler Maunu \hfill maunut@mit.edu \\
	Philippe Rigollet \hfill rigollet@mit.edu \\
	Austin J.\ Stromme \hfill astromme@mit.edu}



	\address{{Department of Mathematics} \\
		{Massachusetts Institute of Technology}\\
		{77 Massachusetts Avenue,}\\
		{Cambridge, MA 02139-4307, USA}
	}
	


\runauthor{Chewi et al.}

\begin{abstract}
    Motivated by the problem of sampling from ill-conditioned log-concave distributions, we give a clean non-asymptotic convergence analysis of mirror-Langevin diffusions as introduced in~\cite{zhang2020mirror}. As a special case of this framework, we propose a class of diffusions called Newton-Langevin diffusions and prove that they converge to stationarity exponentially fast with a rate which not only is dimension-free, but also has no dependence on the target distribution. We give an application of this result to the problem of sampling from the uniform distribution on a convex body using a strategy inspired by interior-point methods. Our general approach follows the recent trend of linking sampling and optimization and highlights the role of the chi-squared divergence.
    In particular, it yields new results on the convergence of the vanilla Langevin diffusion in Wasserstein distance.
\end{abstract}



\end{frontmatter}

\section{Introduction}

Sampling from a target distribution is a central task in statistics and machine learning with applications ranging from Bayesian inference~\cite{RobCas04, durmus2019high} to deep generative models~\cite{Gans14}. Owing to a firm mathematical grounding in the theory of Markov processes~\cite{MeyTwe09}, as well as its great versatility, Markov Chain Monte Carlo (MCMC) has emerged as a fundamental sampling paradigm. While  traditional theoretical analyses are anchored in the asymptotic framework of ergodic theory, this work focuses on  finite-time results that better witness the practical performance of MCMC for high-dimensional problems arising in machine learning.


This perspective parallels an earlier phenomenon in the much better understood field of {optimization} where convexity has played a preponderant role for both theoretical and methodological advances~\cite{nesterov2004optimization, bubeck2015convex}. In fact, sampling and optimization share deep conceptual connections that have contributed to a renewed understanding of the theoretical properties of sampling algorithms~\cite{Dal17, Wib18} building on the seminal work of Jordan, Kinderlehrer and Otto~\cite{jordan1998variational}.

We consider the following canonical sampling problem. Let $\pi$ be a log-concave probability measure over $\R^d$ so that $\pi$ has density equal to $e^{-V}$, where the potential $V:\R^d \to \R$ is convex. Throughout this paper, we also assume that $V$ is twice continuously differentiable for convenience, though many of our results hold under weaker conditions.

Most MCMC algorithms designed for this problem are based on the \emph{Langevin diffusion} (LD), that is the solution ${(X_t)}_{t\ge 0}$ to the stochastic differential equation (SDE)
\begin{align}\label{eq:langevin_sde}
    \D X_t
    &= - \nabla V(X_t) \, \D t + \sqrt 2 \, \D B_t, \tag{$\msf{LD}$}
\end{align}
with ${(B_t)}_{t\ge 0}$ a standard Brownian motion in $\R^d$. Indeed, 
$\pi$ is the unique invariant distribution of~\eqref{eq:langevin_sde} and 
suitable discretizations 
result in algorithms that can be implemented when $V$ is known only up to an additive constant, which is crucial for applications in Bayesian  statistics and machine learning. 

A first connection between sampling from log-concave measures and optimizing convex functions is easily seen from~\eqref{eq:langevin_sde}: omitting the Brownian motion term yields the gradient flow $\dot x_t = - \nabla V(x_t)$, which results in the celebrated gradient descent algorithm when discretized in time~\cite{Dal17, dalalyan2017theoretical}. There is, however, a much deeper connection involving the distribution of $X_t$ rather than $X_t$ itself, and this latter connection has been substantially more fruitful: the marginal distribution of a Langevin diffusion process ${(X_t)}_{t\ge 0}$ evolves according to a \emph{gradient flow}, over the Wasserstein space of probability measures, that minimizes the Kullback-Leibler (KL) divergence $\KL(\cdot \mmid \pi)$~\cite{jordan1998variational,ambrosio2008gradient, villani2009ot}. This point of view has led not only to a better theoretical understanding of the Langevin diffusion~\cite{Ber18, cheng2018langevin, Wib18, durmus2019lmcconvex, vempala2019langevin} but it has also
inspired new sampling algorithms based on classical optimization algorithms, such as proximal/splitting methods~\cite{Ber18, Wib18, Wib19prox,salim2020proximal}, mirror descent~\cite{hsieh2018mirrored, zhang2020mirror}, Nesterov's accelerated gradient descent~\cite{cheng2017underdamped, ma2019there, DalRiou2020}, and Newton methods~\cite{martin2012newtonmcmc, simsekli2016quasinewtonlangevin, wang2020informationnewton}.

\begin{wrapfigure}[22]{r}{.35\textwidth}
    \vspace{-0.2cm}
    \begin{center}
    \includegraphics[width = 0.29\textwidth]{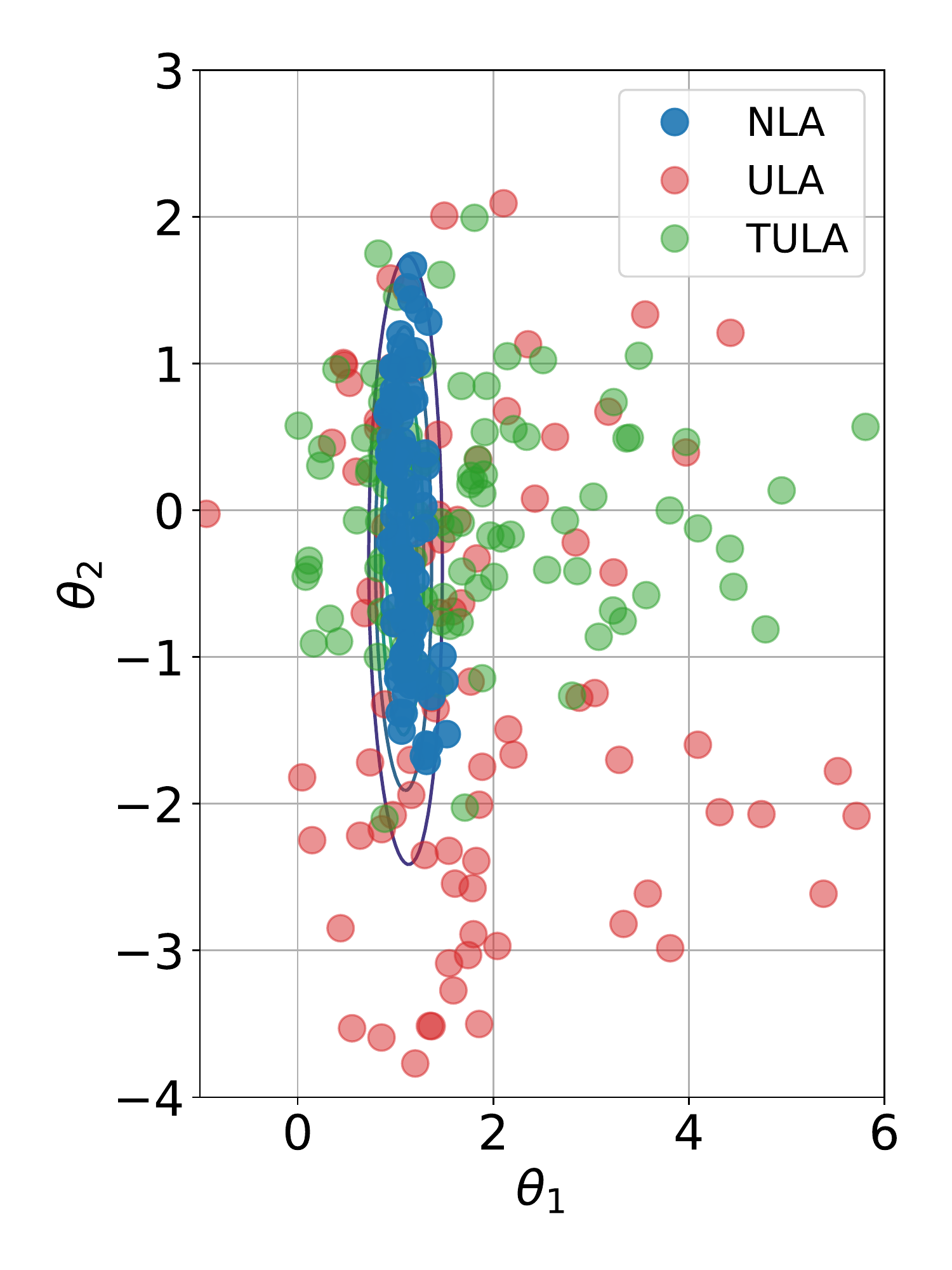}
    \end{center}
    \vspace{-.6cm}
    \caption{Samples from the posterior distribution of a 2D Bayesian logistic regression model using the Newton-Langevin Algorithm \eqref{eq:NLA}, the Unadjusted Langevin Algorithm (ULA), and the Tamed Unadjusted Langevin Algorithm (TULA)~\cite{brosse2019tamed}. For details, see Section~\ref{subsec:lrdetails}.}
    \label{fig:logistic}
\end{wrapfigure}

\medskip
\noindent{\bf Our contributions.}
This paper further exploits the optimization perspective on sampling by establishing a theoretical framework for a large class of stochastic processes called \emph{mirror-Langevin diffusions}~\eqref{eq:mld} introduced in~\cite{zhang2020mirror}. These processes correspond to alternative optimization schemes that minimize the KL divergence over the Wasserstein space by changing its geometry. They show better dependence in key parameters such as the condition number and the dimension.

Our theoretical analysis is streamlined by a technical device which is unexpected at first glance, yet proves to be elegant and effective: we track the progress of these schemes not by measuring the objective function itself, the KL divergence, but rather by measuring the chi-squared divergence to the target distribution $\pi$ as a surrogate. This perspective highlights the central role of mirror Poincar\'e inequalities~\eqref{eq:mp} as sufficient conditions for exponentially fast convergence of the mirror-Langevin diffusion to stationarity in chi-squared divergence, which readily yields convergence in other well-known information divergences, such as the Kullback-Leibler divergence, the Hellinger distance, and the total variation distance~\cite[\S 2.4]{tsybakov2009nonparametric}.

We also specialize our results to the case when the mirror map equals the potential $V$. This can be understood as the sampling analogue of Newton's method, and we therefore call it the \emph{Newton-Langevin diffusion}~\eqref{eq:newton}. In this case, the mirror Poincar\'e inequality translates into the Brascamp-Lieb inequality which automatically holds when $V$ is twice-differentiable and \emph{strictly} convex. In turn, it readily implies exponential convergence of the Newton-Langevin diffusion (Corollary~\ref{cor:conv_newton}) and can be used for approximate sampling even when the second derivative of $V$ vanishes (Corollary~\ref{cor:degenerate_sampling}). Strikingly, the rate of convergence \emph{has no dependence on $\pi$ or on the dimension $d$} and, in particular, is robust to cases where $\nabla^2 V$ is arbitrarily close to zero. This \emph{scale-invariant} convergence parallels that of Newton's method in convex optimization and is the first result of this kind for sampling.

This invariance property is useful for approximately sampling from the uniform distribution over a
    convex body $\eu C$, which has been well-studied in the computer science literature \cite{FriKanPol94, kls1995, LovVem07}. By taking the target distribution $\pi \propto \exp(-\beta V)$, where $V$ is any strictly convex \emph{barrier function}, and $\beta$, the inverse temperature parameter, is taken to be small (depending on the target accuracy), we can use the Newton-Langevin diffusion, much in the spirit of interior point methods (as promoted by~\cite{LadTatVem20}), to output a sample which is approximately uniformly distributed on $\eu C$; see Corollary~\ref{cor:unif_sampling}.



Throughout this paper, we work exclusively in the setting of continuous-time diffusions such as~\eqref{eq:langevin_sde}.
We refer to the works~\cite{DurMou15, Dal17, dalalyan2017theoretical, raginsky2017sgld, cheng2018langevin, Wib18, DalKar17, durmus2019lmcconvex, dalalyan2019bounding, mou2019improved, vempala2019langevin} for discretization error bounds, and leave this question open for future works.


\medskip
\noindent{\bf Related work.} The discretized Langevin algorithm, and the Metropolis-Hastings adjusted version, have been well-studied when used to sample from strongly log-concave distributions, or distributions satisfying a log-Sobolev inequality \cite{dalalyan2017theoretical, durmus2017nonasymptoticlangevin, cheng2018langevin, cheng2019quantitative, DalKar17, durmus2019high, dwivedi2019log, mou2019improved, vempala2019langevin}. 
Moreover, various ways of adapting Langevin diffusion to sample from bounded domains have been proposed \cite{bubeck2018sampling, hsieh2018mirrored, zhang2020mirror}; in particular, \cite{zhang2020mirror} studied the discretized mirror-Langevin diffusion.
Finally, we note that while our analysis and methods are inspired by the optimization perspective on sampling, it connects to a more traditional analysis based on coupling stochastic processes. Quantitative analysis of the continuous Langevin diffusion process associated to SDE~\eqref{eq:langevin_sde} has been performed with Poincar\'e and log-Sobolev inequalities \cite{bolley2012convergence, bakry2014markov,vempala2019langevin}, and with couplings of stochastic processes \cite{chen1989coupling, eberle2016reflection}.

\medskip 

\noindent{\bf Notation.} The Euclidean norm over $\R^d$ is denoted by $\|\cdot\|$. Throughout, we simply write $\int \!g$ to denote the integral with respect to the Lebesgue measure: $\int \!g(x) \,\D x$. When the integral is with respect to a different measure $\mu$, we explicitly write $\int \!g \, \D \mu$. The expectation and variance of $g(X)$ when $X \sim \mu$ are respectively denoted $\E_\mu g=\int\! g \, \D \mu$ and $\var_\mu g:= \int \! {(g-\E_\mu g)}^2 \, \D \mu$.  
When clear from context, we sometimes abuse notation by identifying a measure $\mu$ with its Lebesgue density.

\section{Mirror-Langevin diffusions}

Before introducing mirror-Langevin diffusions, our main objects of interest, we provide some intuition for their construction by drawing a parallel with convex optimization.

\subsection{Gradient flows, mirror flows, and Newton's method}
\label{sec:gradientflows}
We briefly recall some background on gradient flows and mirror flows; we refer readers to the monograph~\cite{bubeck2015convex} for the convergence analysis of the corresponding discrete-time algorithms.

Suppose we want to minimize a differentiable function $f : \R^d\to\R$. The \emph{gradient flow} of $f$ is the curve $(x_t)_{t \ge 0}$ on $\R^d$ solving $\dot x_t = -\nabla f(x_t)$. A suitable time discretization of this curve yields the well-known  \emph{gradient descent} (GD).


Although the gradient flow typically works well for optimization over Euclidean spaces, it may suffer from poor dimension scaling  in more general cases such as Banach space optimization; a notable example is the case when $f$ is defined over the probability simplex equipped with the $\ell_1$ norm. This observation led Nemirovskii and Yudin~\cite{nemirovskii1979problemcomplexity} to introduce the \emph{mirror flow}, which is defined as follows. Let $\phi : \R^d\to\R\cup\{\infty\}$ be a \emph{mirror map}, that is a strictly convex twice continuously differentiable function of \emph{Legendre type}\footnote{This ensures that $\nabla \phi$ is invertible, c.f.~\cite[\S 26]{rockafellar1997convex}.}.
The mirror flow $(x_t)_{t \ge 0}$ satisfies $\partial_t \nabla \phi(x_t) = -\nabla f(x_t)$, or equivalently, $\dot x_t = - {[\nabla^2 \phi(x_t)]}^{-1} \nabla f(x_t)$. The corresponding discrete-time algorithms, called \emph{mirror descent}  (MD) algorithms, have been successfully employed in varied tasks of machine learning~\cite{bubeck2015convex} and online optimization~\cite{BubCes12} where the entropic mirror map plays an important role. In this work, we are primarily concerned with the following choices for the mirror map:
%

\begin{enumerate}
    \item When $\phi = \|\cdot\|^2/2$, then the mirror flow reduces to the gradient flow.
    \item Taking $\phi = f$ and the discretization $x_{k+1} = x_k - h_k\, {[\nabla^2 f(x_k)]}^{-1} \nabla f(x_k)$  yields another popular optimization algorithm known as (damped) \emph{Newton's method}.
    Newton's method has the important property of being invariant under affine transformations of the problem, and its local convergence is known to be much faster than that of GD; see \cite[\S 5.3]{bubeck2015convex}. 
\end{enumerate}

\subsection{Mirror-Langevin diffusions}

We now introduce the \emph{mirror-Langevin diffusion} (MLD) of~\cite{zhang2020mirror}.
Just as \ref{eq:langevin_sde} corresponds to the gradient flow, the MLD is the sampling analogue of the mirror flow.
To describe it, let $\phi : \R^d \to \R$ be a mirror map as in the previous section.
Then, the mirror-Langevin diffusion satisfies the SDE
\begin{align}\label{eq:mld} \tag{$\msf{MLD}$}
    X_t = \nabla \phi^\star(Y_t),\qquad\qquad\D Y_t
    = -\nabla V(X_t) \, \D t + \sqrt 2 \, {[\nabla^2 \phi(X_t)]}^{1/2} \, \D B_t\,,
\end{align}
where $\phi^\star$ denotes the convex conjugate of $\phi$~\cite[\S 3.3]{borwein2006convex}.
In particular, if we choose the mirror map $\phi$ to equal the potential $V$, then we arrive at a sampling analogue of \emph{Newton's method}, which we call the \emph{Newton-Langevin diffusion} (NLD),
\begin{align}\label{eq:newton} \tag{$\msf{NLD}$}
    X_t = \nabla V^\star(Y_t),\qquad\qquad\D Y_t
    = -\nabla V(X_t) \, \D t + \sqrt 2 \, {[\nabla^2 V(X_t)]}^{1/2} \, \D B_t.
\end{align}
From our intuition gained from optimization, we expect that \ref{eq:newton} has special properties, such as affine invariance and faster convergence. We validate this intuition in Corollary~\ref{cor:conv_newton} below by showing that, provided $\pi$ is strictly log-concave, the \ref{eq:newton} converges to stationarity exponentially fast, with no dependence on $\pi$. This should be contrasted with the vanilla Langevin diffusion~\eqref{eq:langevin_sde}, for which the convergence rate depends on the Poincar\'e constant of $\pi$, as we discuss in the next section.


We end this section by comparing \ref{eq:mld} and \ref{eq:newton} with similar sampling algorithms proposed in the literature inspired by mirror descent and Newton's method.

\medskip
\noindent{\textit{Mirrored Langevin dynamics}}. A variant of \ref{eq:mld}, called {``mirrored Langevin dynamics"}, was introduced in~\cite{hsieh2018mirrored}. The mirrored Langevin dynamics is motivated by constrained sampling and corresponds to the vanilla Langevin algorithm applied to the new target measure ${(\nabla \phi)}_\# \pi$.
In contrast, \ref{eq:mld} can be understood as a Riemannian diffusion w.r.t.\ the Riemannian metric induced by the mirror map $\phi$. Thus, the motivations and properties of the two algorithms are different, and we refer to~\cite{zhang2020mirror} for further comparison of the two algorithms.

\medskip

\noindent{\textit{Quasi-Newton diffusion}}. The paper~\cite{simsekli2016quasinewtonlangevin} proposes a quasi-Newton sampling algorithm, based on L-BFGS, which is partly motivated by the desire to avoid computation of the third derivative $\nabla^3 V$ while implementing the Newton-Langevin diffusion. We remark, however, that the form of \ref{eq:newton} employed above, which treats $V$ as a mirror map, does not in fact require the computation of $\nabla^3 V$, and thus can be implemented practically; see Section~\ref{sec:sim}. Moreover, since we analyze the full \ref{eq:newton}, rather than a quasi-Newton implementation, we are able to give a clean convergence result.

\medskip

\noindent{\textit{Information Newton's flow}}. Inspired by the perspective of~\cite{jordan1998variational}, which views the Langevin diffusion as a gradient flow in the Wasserstein space of probability measures, the paper~\cite{wang2020informationnewton} proposes an approach termed {``information Newton's flow"} that applies Newton's method directly on the space of probability measures equipped with either the Fisher-Rao or the Wasserstein metric. However, unlike~\ref{eq:langevin_sde} and~\ref{eq:newton} that both operate at the level of particles, 
information Newton's flow faces significant challenges at the level of both implementation and analysis.

\section{Convergence analysis}

\subsection{Convergence of gradient flows and mirror flows}\label{scn:conv_grad_flows}

We provide a brief reminder about the convergence analysis of gradient flows and mirror flows defined in Section~\ref{sec:gradientflows} to provide intuition for the next section. Throughout, let $f$ be a differentiable function with minimizer $x^*$.

Consider first the gradient flow for $f$: $\dot{x_t}=-\nabla f(x_t)$. We get $\partial_t [f(x_t) - f(x^*)] = - \|\nabla f(x_t)\|^2$ from a straightforward computation. From this identity, it is natural to assume a  \emph{Polyak-{\L}ojasiewicz {\rm (PL)} inequality}, which is well-known in the optimization literature~\cite{karimi2016linear} and can be employed even when $f$ is not convex~\cite{CheMauRig20}. Indeed, if there exists a constant $C_{\msf{PL}} > 0$ with
\begin{align}\label{eq:pl}\tag{$\msf{PL}$}
 f(x) - f(x^*) \le     \frac{C_{\msf{PL}}}{2} \, \norm{\nabla f(x)}^2  \qquad \forall\, x \in \R^d\,,
\end{align}
then $\partial_t [f(x_t) - f(x^*)] \le -  \frac{2}{C_{\msf{PL}}} \, [f(x_t) - f(x^*)]$. Together with  Gr\"onwall's inequality, it readily yields exponentially fast convergence in objective value: $f(x_t) \le f(x_0) \, \e^{-2t/C_{\msf{PL}}}$.

A similar analysis may be carried out for the mirror flow. Fix a mirror map $\phi$ and consider the mirror flow: $\dot x_t = - {[\nabla^2 \phi(x_t)]}^{-1} \nabla f(x_t)$. It holds $\partial_t [f(x_t) - f(x^*)] = - \langle \nabla f(x_t), {[\nabla^2 \phi(x_t)]}^{-1} \nabla f(x_t) \rangle$. Therefore, the analogue of~\eqref{eq:pl} which guarantees exponential decay in the objective value is the following inequality, which we call a \emph{mirror PL inequality}:
\begin{align}\label{eq:mpl} \tag{$\msf{MPL}$}
    f(x) - f(x^*) \le  \frac{C_{\msf{MPL}}}{2} \, \langle \nabla f(x), {[\nabla^2 \phi(x)]}^{-1} \nabla f(x) \rangle  \qquad \forall x \in \R^d.
\end{align}
Next, we describe analogues of~\eqref{eq:pl} and~\eqref{eq:mpl} that guarantee convergence of \ref{eq:langevin_sde} and~\ref{eq:mld}.

\subsection{Convergence of mirror-Langevin diffusions}
\label{sec:mld}
The above analysis employs the objective function $f$ to measure the progress of both the gradient and mirror flows. While this is the most natural choice, our approach below crucially relies on measuring progress via a \emph{different functional} $F$.
What should we use as $F$? To answer this question, we first consider the simpler case of the vanilla Langevin diffusion~\eqref{eq:langevin_sde}, which is a special case of~\ref{eq:mld} when the mirror map is $\phi = \norm\cdot^2/2$. We keep this discussion informal and postpone rigorous arguments to Appendix~\ref{appendix:markov_semigroup}.

Since the work of~\cite{jordan1998variational}, it has been known that the marginal distribution $\mu_t$ at time $t\ge 0$ of \ref{eq:langevin_sde} evolves according to the \emph{gradient flow} of the KL divergence $D_{\rm KL}(\cdot \mmid \pi)$ with respect to the $2$-Wasserstein distance $W_2$; we refer readers to~\cite{santambrogio2017euclidean} for an overview of this work, and to~\cite{ambrosio2008gradient, villani2009ot} for comprehensive treatments. Therefore, the most natural choice for $F$ is, as in Section~\ref{scn:conv_grad_flows}, the objective function $D_{\rm KL}(\cdot\mmid \pi)$ itself. Following this approach, one can compute~\cite[\S 9.1.5]{villani2003topics}
\begin{align*}
\partial_t D_{\rm KL}(\mu_t \mmid \pi) = - \int  \bigl\lVert \nabla\ln \frac{\D \mu_t}{\D \pi} \bigr\rVert^2 \, \D \mu_t = -4\int \bigl\lVert \nabla \sqrt{\frac{\D \mu_t}{\D \pi}} \bigr\rVert^2 \, \D \pi.
\end{align*}
In this setup, the role of the PL inequality~\eqref{eq:pl} is played by a \emph{log-Sobolev inequality} of the form
\begin{align}\label{eq:lsi}\tag{$\msf{LSI}$}
    \ent_\pi(g^2)
    &:= \int g^2 \ln(g^2) \, \D \pi  - \big(\int g^2 \, \D\pi\big) \ln \big(\int g^2 \, \D\pi\big)
    \le 2C_{\msf{LSI}} \int \norm{\nabla g}^2 \, \D \pi.
\end{align}
When $g=\sqrt{\D \mu_t/\D \pi}$,~\eqref{eq:lsi} reads 
$
D_{\rm KL}(\mu_t \mmid \pi) \le -2C_{\msf{LSI}}^{-1} \int \bigl\lVert \nabla \sqrt{\D\mu_t/\D\pi} \bigr\rVert^2\, \D \pi\,,$
which implies exponentially fast convergence: $D_{\rm KL}(\mu_t \mmid \pi)
    \le D_{\rm KL}(\mu_0 \mmid \pi) \, \e^{-2t/C_{\msf{LSI}}}$ by Gr\"onwall's inequality.

A disadvantage of this approach, however, is that the log-Sobolev inequality~\eqref{eq:lsi} does not hold for any log-concave measure $\pi$, or it may hold with a poor constant $C_{\msf{LSI}}$. For example, it is known that the log-Sobolev constant of an isotropic log-concave distribution must in general depend on the diameter of its support~\cite{leevempala2018lsi}. In contrast, we work below with a \emph{Poincar\'e inequality}, which is conjecturally satisfied by such distributions with a \emph{universal constant}~\cite{kls1995}.


Motivated by~\cite{bakryRateConvergenceErgodic2008, cattiaux2009trendtv},
we instead consider the \emph{chi-squared divergence}
\begin{align*}
    F(\mu) = \chi^2(\mu \mmid \pi) := \var_\pi \frac{\D\mu}{\D\pi}=\int \Big(\frac{\D \mu}{\D \pi}\Big)^2\D \pi -1, \qquad \text{if}~\mu \ll \pi\,,
\end{align*}
and $F(\mu)=\infty$ otherwise.
It is well-known that the law ${(\mu_t)}_{t\ge 0}$ of~\ref{eq:langevin_sde} satisfies the Fokker-Planck equation in the weak sense~\cite[\S 5.7]{karatzasBrownianMotion1998}:   
    \begin{align*}
        \partial_t \mu_t
        &= \divergence\bigl(\mu_t \, \nabla \ln \frac{\mu_t}{\pi}\bigr).
    \end{align*}
    Using this, we can compute the derivative of the chi-squared divergence:
    \begin{align*}
        \frac{1}{2}\partial_t F(\mu_t)
        &= \int \! \frac{\mu_t}{\pi} \, \partial_t \mu_t
        = \int \! \frac{\mu_t}{\pi} \divergence\bigl(\mu_t \nabla \ln \frac{\mu_t}{\pi}\bigr)
        = -\int \bigl\langle  \nabla \ln \frac{\mu_t}{\pi}, \nabla \frac{\mu_t}{\pi} \bigr\rangle \, \mu_t 
        = - \int \! \big\| \nabla \frac{\mu_t}{\pi}\big\|^2 \, \pi\,,
    \end{align*}
and exponential convergence of the chi-squared divergence follows if $\pi$ satisfies a Poincar\'e inequality:
\begin{equation}\label{eq:p}\tag{$\msf{P}$}
 \var_\pi g  \le  C_{\msf{P}} \E_\pi[\|\nabla g\|^2]  \qquad\text{for all locally Lipschitz}~g \in L^2(\pi).
\end{equation}
Thus, when using the chi-squared divergence to track progress, the role of the PL inequality is played by a Poincar\'e inequality. As we discuss in Sections~\ref{scn:newton_conv} and~\ref{scn:langevin_conv} below, the Poincar\'e inequality is significantly weaker than the log-Sobolev inequality.

A similar analysis may be carried out for \ref{eq:mld} using an appropriate variation of Poincar\'e inequalities.
\begin{defn}[Mirror Poincar\'e inequality]
Given a mirror map $\phi$, we say that the distribution $\pi$ satisfies a \emph{mirror Poincar\'e inequality}  with constant $C_{\msf{MP}}$ if
\begin{align}\label{eq:mp} \tag{$\msf{MP}$}
      \var_\pi g\le  C_{\msf{MP}} \E_\pi\langle \nabla g, {(\nabla^2 \phi)}^{-1} \nabla g \rangle
   \qquad \text{for all locally Lipschitz}~g \in L^2(\pi).
\end{align}
When $\phi=\|\cdot\|^2/2$,~\eqref{eq:mp} is simply called a \emph{Poincar\'e inequality} and the smallest  $C_{\msf{MP}}$ for which the inequality holds is the \emph{Poincar\'e constant} of $\pi$, denoted $C_{\msf{P}}$.
\end{defn}
Using a similar argument as the one above, we show exponential convergence of~\ref{eq:mld} in chi-squared divergence $\chi^2(\cdot \mmid \pi)$ under~\eqref{eq:mp}. Together with standard comparison inequalities between information divergences~\cite[\S 2.4]{tsybakov2009nonparametric}, it implies exponential convergence in a variety of commonly used divergences, including the total variation (TV) distance $\norm{\cdot -\pi}_{\rm TV}$, the Hellinger distance $H(\cdot, \pi)$, and the KL divergence $D_{\rm KL}(\cdot \mmid \pi)$.

\begin{thm} \label{thm:exp_chi_2}
For each $t\ge 0$, let $\mu_t$ be the marginal distribution of \ref{eq:mld}  with target distribution $\pi$ at time $t$. Then if $\pi$ satisfies the mirror Poincar\'e inequality~\eqref{eq:mp}
with constant $C_{\msf{MP}}$, it holds
\[
2 \norm{\mu_t - \pi}_{\rm TV}^2, \ H^2(\mu_t,\pi),\  D_{\rm KL}(\mu_t\mmid \pi),\ \chi^2(\mu_t\mmid \pi)\le \e^{-\frac{2t}{C_{\msf{MP}}}}\chi^2(\mu_0\mmid \pi), \quad \forall \, t \ge 0\,.
\]
\end{thm}
We give two proofs of this result in Appendix~\ref{appendix:markov_semigroup}.

Recall that \ref{eq:langevin_sde} can be understood as a gradient flow for the KL divergence on the 2-Wasserstein space. In light of this interpretation, the above bound for the KL divergence yields a convergence rate \emph{in objective value}, and it is natural to wonder whether a similar rate holds for the iterates themselves in terms of 2-Wasserstein distance. From the works~\cite{ding2015quadratictransport, ledoux2018remarks, liu2020quadratictransport}, it is known that a Poincar\'e inequality~\eqref{eq:p} implies the transportation-cost inequality
\begin{align}\label{eq:liu_transport}
    W_2^2(\mu,\pi) \le 2C_{\msf P} \chi^2(\mu \mmid \pi), \qquad \forall \mu \ll \pi.
\end{align}
Initially unaware of these works, we proved that a Poincar\'e inequality implies a suboptimal chi-squared transportation inequality. Since the suboptimal inequality already suffices for our purposes, we state and prove it in Appendix~\ref{scn:wasserstein}. We thank Jon Niles-Weed for bringing this to our attention.

The inequality~\eqref{eq:liu_transport} implies that if $\pi$ has a finite Poincar\'e constant $C_{\msf P}$ then Theorem~\ref{thm:exp_chi_2} also yields exponential convergence in Wasserstein distance. In the rest of the paper, we write this as
$$
\frac{1}{2 C_{\msf P}} W_2^2(\mu_t,\pi) \le \e^{-\frac{2t}{C_{\msf{MP}}}} \chi^2(\mu_0 \mmid \pi)\,,
$$
for \emph{any} target measure $\pi$ that satisfies a mirror Poincar\'e inequality, with the convention that $C_{\msf P}=\infty$ when $\pi$ fails to satisfy a Poincar\'e inequality. In this case, the above inequality is simply vacucous.

\section{Applications}

We specialize Theorem~\ref{thm:exp_chi_2} to the following important applications.

\subsection{Newton-Langevin diffusion}\label{scn:newton_conv}
\label{subsec:newton}
For~\ref{eq:newton}, we assume that $V$ is strictly convex and twice continuously differentiable; take $\phi=V$. In this case, the mirror Poincar\'e inequality~\eqref{eq:mp} reduces to the \emph{Brascamp-Lieb inequality}, which is known to hold with constant $C_{\msf{MP}} = 1$ for any strictly log-concave distribution~$\pi$~\cite{BraLie76, bobkov2000prekopa,gentil2008prekopa}. It yields the following remarkable result where the exponential contraction rate has no dependence on $\pi$ nor on the dimension $d$.


\begin{cor}\label{cor:conv_newton}
    Suppose that $V$ is strictly convex and  twice continuously differentiable. Then, the law ${(\mu_t)}_{t\ge 0}$ of~\ref{eq:newton} satisfies
    \begin{align*}
       2 \norm{\mu_t - \pi}_{\rm TV}^2,\ H^2(\mu_t,\pi),\ D_{\rm KL}(\mu_t \mmid \pi),\ \chi^2(\mu_t\mmid \mu),\ \frac{1}{2 C_{\msf P}} W_2^2(\mu_t,\pi) \le \e^{-2t} \chi^2(\mu_0 \mmid \pi).
    \end{align*}
\end{cor} 

\begin{wrapfigure}[13]{r}{.45\textwidth}
    \vspace{-0.3cm}
    \begin{center}
    \includegraphics[width = 0.35\textwidth]{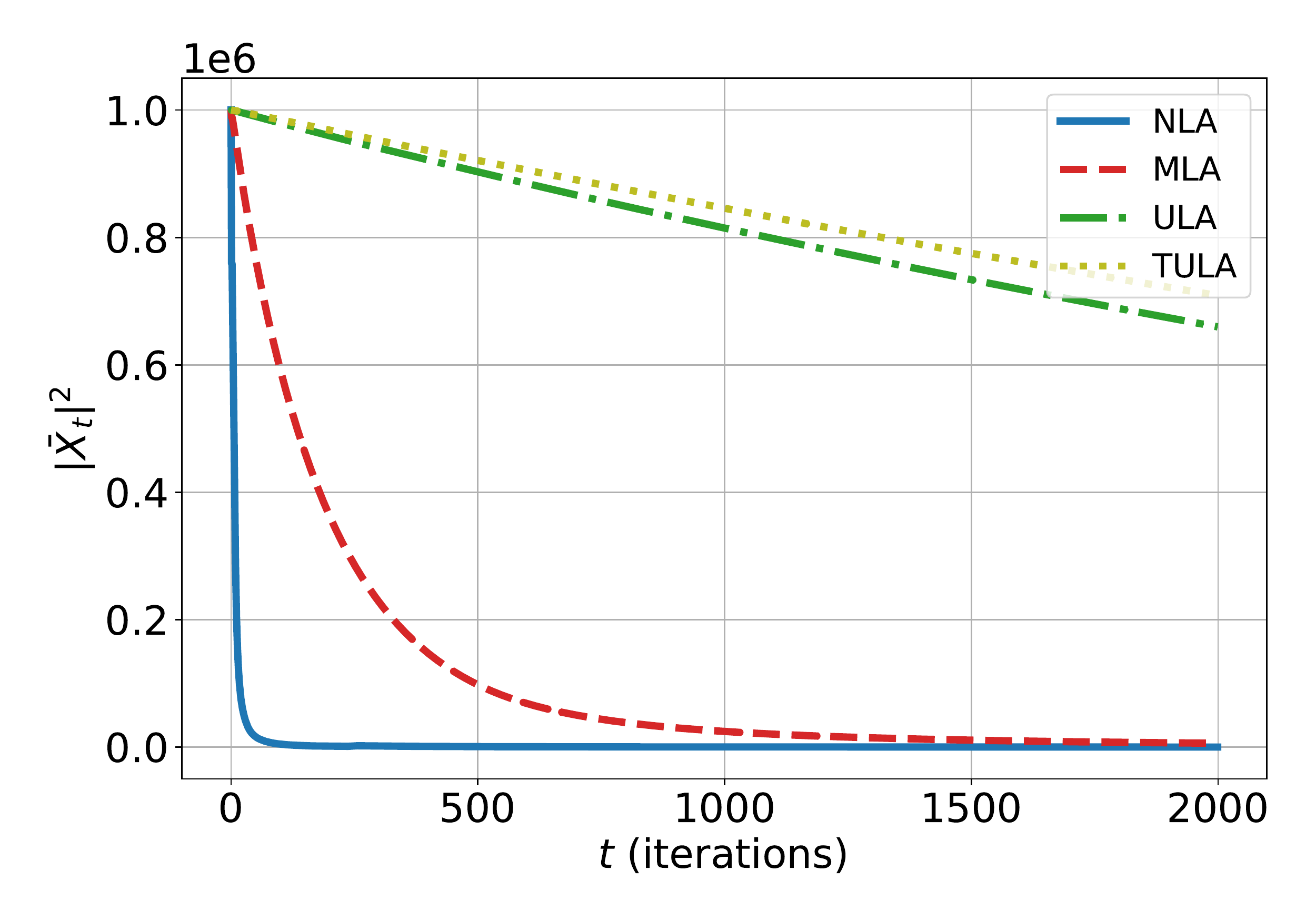}
    \end{center}
    \vspace{-.6cm}
    \caption{Approximately sampling from $\pi \propto \e^{-\|\cdot\|}$ by sampling from $\pi_\beta \propto \e^{-\|\cdot\| - \beta\|\cdot-\mathbf 1\|^2}$ ($\beta=.0005$). Algorithms are initialized at a random $X_0$ with $\|X_0\| = 1000$. The plot shows the squared distance of the running means to $0$.}
    \label{fig:lapfull2}
\end{wrapfigure}

If $\pi$ is log-concave,
then it satisfies a Poincar\'e inequality~\cite{alonso2015kls,leevempala2017poincare} so that the result in Wasserstein distance holds. In fact, contingent on the famous \emph{Kannan-Lov\'asz-Simonovitz} (KLS) conjecture \cite{kls1995}, the Poincar\'e constant of any log-concave distribution $\pi$ is upper bounded by a constant, independent of the dimension, times the largest eigenvalue of the covariance matrix of $\pi$.

 At this point, one may wonder, under the same assumptions as the Brascamp-Lieb inequality, whether a mirror version of the log-Sobolev inequality~\eqref{eq:lsi} holds. This question was answered negatively in~\cite{bobkov2000prekopa}, thus reinforcing our use of the chi-squared divergence as a surrogate for the KL divergence.

If the potential $V$ is convex, but degenerate (i.e., not strictly convex) we cannot use \ref{eq:newton} directly with $\pi$ as the target distribution. Instead, we perturb $\pi$ slightly to a new measure $\pi_\beta$, which is strongly log-concave, and for which we can use \ref{eq:newton}. Crucially, due to the scale invariance of \ref{eq:newton}, the time it takes for \ref{eq:newton} to mix does not depend on $\beta$, the parameter which governs the approximation error.

\begin{cor}\label{cor:degenerate_sampling}
    Fix a target accuracy $\varepsilon > 0$.
    Suppose $\pi = \e^{-V}$ is log-concave and set $\pi_\beta \propto \e^{-V- \beta \|\cdot\|^2}$, where $\beta \le \varepsilon^2/(2\int \|\cdot\|^2 \, \D \pi)$. Then, the law ${(\mu_t)}_{t\ge 0}$ of \ref{eq:newton} with target distribution $\pi_\beta$ satisfies $\norm{\mu_t - \pi}_{\rm TV}\le \eps$ by time $t = \frac{1}{2} \ln[2\chi^2(\mu_0 \mmid \pi_\beta)] + \ln(1/\varepsilon)$.
\end{cor}
\begin{proof}
    From our assumption, it holds
    \begin{align*}
        D_{\rm KL}(\pi \mmid \pi_\beta)
       = \int \ln \frac{\D\pi}{\D\pi_\beta} \, \D \pi
       = \beta \int \|\cdot\|^2 \, \D\pi + \ln \int \e^{-\beta\|\cdot\|^2} \, \D \pi
        \le \beta \int\|\cdot\|^2 \, \D \pi
        \le \frac{\varepsilon^2}{2}.
    \end{align*}
    Moreover, Theorem~\ref{thm:exp_chi_2} with the above choice of $t$ yields $D_{\rm KL}(\mu_t \mmid \pi_\beta) \le \eps^2/2$. To conclude, we use Pinsker's inequality and the triangle inequality for $\|\cdot\|_{\rm TV}$.
\end{proof}

Convergence guarantees for other cases where $\phi$ is only a \emph{proxy} for $V$  are presented in Appendix~\ref{app:ext}. 

\subsection{Sampling from the uniform distribution on a convex body}\label{scn:unif_sampling}

Next, we consider an application of~\ref{eq:newton} to the problem of sampling from the uniform distribution $\pi$ on a convex body $\eu C$. A natural method of outputting an approximate sample from $\pi$ is to take a strictly convex function $\widetilde V : \R^d\to\R \cup \{\infty\}$ such that $\dom \widetilde V = \eu C$ and $\widetilde V(x) \to \infty$ as $x\to\partial \eu C$, and to run~\ref{eq:newton} with target distribution $\pi_\beta \propto \e^{-\beta \widetilde V}$, where the inverse temperature $\beta$ is taken to be small (so that $\pi_\beta \approx \pi$). The function $\widetilde V$ is known as a \emph{barrier function}.

Although we can take any choice of barrier function $\widetilde V$, we obtain a clean theoretical result if we assume that $\widetilde V$ is $\nu^{-1}$-$\exp$-concave, that is, the mapping $\exp(- \nu^{-1} \widetilde V)$ is concave. Interestingly, this assumption further deepens the rich analogy between sampling and optimization, since such barriers are widely studied in the optimization literature. There, the property of $\exp$-concavity is typically paired with the property of \emph{self-concordance}, and barrier functions satisfying these two properties are a cornerstone of the theory of \emph{interior point algorithms} (see~\cite[\S 5.3]{bubeck2015convex} and~\cite[\S 4]{nesterov2004optimization}).

\begin{wrapfigure}[14]{r}{.20\textwidth}
    \begin{center}
     \includegraphics[width=.18\textwidth]{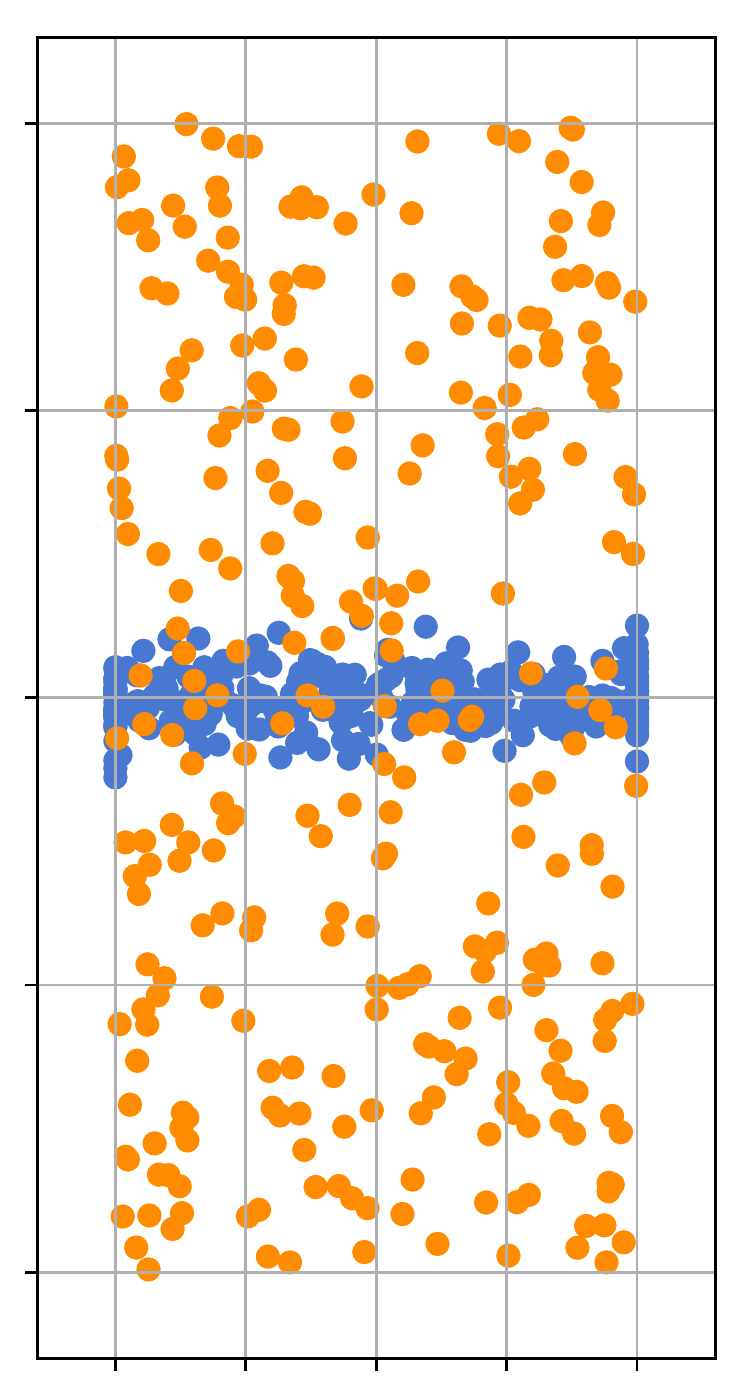}
    \end{center}
    \vspace{-0.5cm}
    \caption{Uniform sampling from the set $[-0.01,0.01]\times[-1,1]$: PLA (blue) vs.\ \ref{eq:NLA} (orange). See Section~\ref{subsec:unif_cvx}.}
    \label{fig:ellipse}
\end{wrapfigure}
We now formulate our sampling result. In our continuous framework, it does not require self-concordance of the barrier function.

\begin{cor}\label{cor:unif_sampling}
    Fix a target accuracy $\varepsilon>0$. Let $\pi$ be the uniform distribution over a convex body $\eu C$ and let $\widetilde V$ be a $\nu^{-1}$-$\exp$-concave barrier for $\eu C$. Then, the law ${(\mu_t)}_{t\ge 0}$ of~\ref{eq:newton} with target density $\pi_\beta \propto \e^{-\beta \widetilde V}$ for $\beta \le \varepsilon^2/(2\nu)$ satisfies
    $\norm{\mu_t - \pi}_{\rm TV} \le \varepsilon$
    by time $t= \frac{1}{2} \ln[2\chi^2(\mu_0 \mmid \pi_\beta)] + \ln (1/\varepsilon)$.
\end{cor}
\begin{proof}
Lemma~\ref{lem:self_concordant} in Appendix~\ref{appendix:auxiliary} ensures that $D_{\rm KL}(\pi_\beta \mmid \pi) \le \varepsilon^2/2$. We conclude as in the proof of Corollary~\ref{cor:degenerate_sampling}, by using Theorem~\ref{thm:exp_chi_2}, Pinsker's inequality, and the triangle inequality for $\|\cdot\|_{\rm TV}$.
\end{proof}


We demonstrate the efficacy of \ref{eq:newton} in a simple simulation: sampling uniformly from the ill-conditioned rectangle $[-a,a]\times[-1,1]$ with $a=0.01$ (Figure~\ref{fig:ellipse}).
We compare \ref{eq:NLA} with the Projected Langevin Algorithm (PLA)~\cite{bubeck2018sampling}, both with 200 iterations and $h=10^{-4}$.
For \ref{eq:NLA}, we take $\widetilde V(x) = -\log (1-x_1^2) - \log (a^2 - x_2^2)$
and $\beta = 10^{-4}$.

\subsection{Langevin diffusion under a Poincar\'e inequality}\label{scn:langevin_conv}

We conclude this section by giving some implications of Theorem~\ref{thm:exp_chi_2} to the classical Langevin diffusion~\eqref{eq:langevin_sde} when $\phi=\|\cdot\|^2/2$. In this case, the mirror Poincar\'e inequality~\eqref{eq:mp} reduces to the classical Poincar\'e inequality~\eqref{eq:p} as in Section~\ref{sec:mld}. 


\begin{cor}
Suppose that $\pi$ satisfies a Poincar\'e inequality~\eqref{eq:p} with constant $C_{\msf P} > 0$. Then, the law ${(\mu_t)}_{t\ge 0}$ of the Langevin diffusion~\eqref{eq:langevin_sde} satisfies
\begin{align*}
2 \norm{\mu_t - \pi}_{\rm TV}^2,\ H^2(\mu_t,\pi),\ D_{\rm KL}(\mu_t \mmid \pi),\ \chi^2(\mu_t\mmid \mu),\ \frac{1}{2 C_{\msf P}} W_2^2(\mu_t,\pi) \le \e^{-\frac{2t}{C_{\msf P}}} \chi^2(\mu_0 \mmid \pi).
\end{align*}
\end{cor}

The convergence in TV distance recovers results of~\cite{dalalyan2017theoretical, durmus2017nonasymptoticlangevin}. Bounds for the stronger error metric $\chi^2(\cdot\mmid \pi)$ have appeared explicitly in \cite{cao2019renyi, vempala2019langevin} and is implicit in the work of~\cite{bakryRateConvergenceErgodic2008, cattiaux2009trendtv} on which the TV bound of~\cite{durmus2017nonasymptoticlangevin} is based.

Moreover, it is classical that if $\pi$ satisfies a log-Sobolev inequality~\eqref{eq:lsi} with constant $C_{\msf{LSI}}$ then it has Poincar\'e constant $C_{\msf P} \le C_{\msf{LSI}}$. Thus, the choice of the chi-squared divergence as a surrogate for the KL divergence when tracking progress indeed requires weaker assumptions on $\pi$.

\section{Numerical experiments}
\label{sec:sim}

In this section, we examine the numerical performance of the \emph{Newton-Langevin Algorithm} (NLA), which is given by the following Euler discretization of~\ref{eq:newton}:
\begin{equation}\label{eq:NLA}\tag{$\msf{NLA}$}
\nabla V(X_{k+1}) = (1 - h)\nabla V(X_k) + \sqrt{2 h} \, {[\nabla^2 V(X_k)]}^{1/2} \xi_k,
\end{equation}
where ${(\xi_k)}_{k\in\N}$ is a sequence of i.i.d.\ $\mathcal{N}(0, I_d)$ variables.
In cases where $\nabla V$ does not have a closed-form inverse, such as the logistic regression case of Section~\ref{subsec:lrdetails}, we invert it numerically by solving the convex optimization problem $\nabla V^{\star}(y) = \argmax_{x \in \R^d}{\{\langle x, y \rangle - V(x)\}}.$ 

We focus here on sampling from an ill-conditioned generalized Gaussian distribution on $\R^{100}$ with $V(x) = \langle x, \Sigma^{-1} x \rangle^\gamma/2$ for $\gamma = 3/4$ to demonstrate the scale invariance of~\ref{eq:newton} established in Corollary~\ref{cor:conv_newton}. Additional experiments, including the Gaussian case $\gamma=1$, are given in Appendix~\ref{appendix:numericals}.


\begin{figure}[h]
    \centering
    \includegraphics[width = 0.49\textwidth]{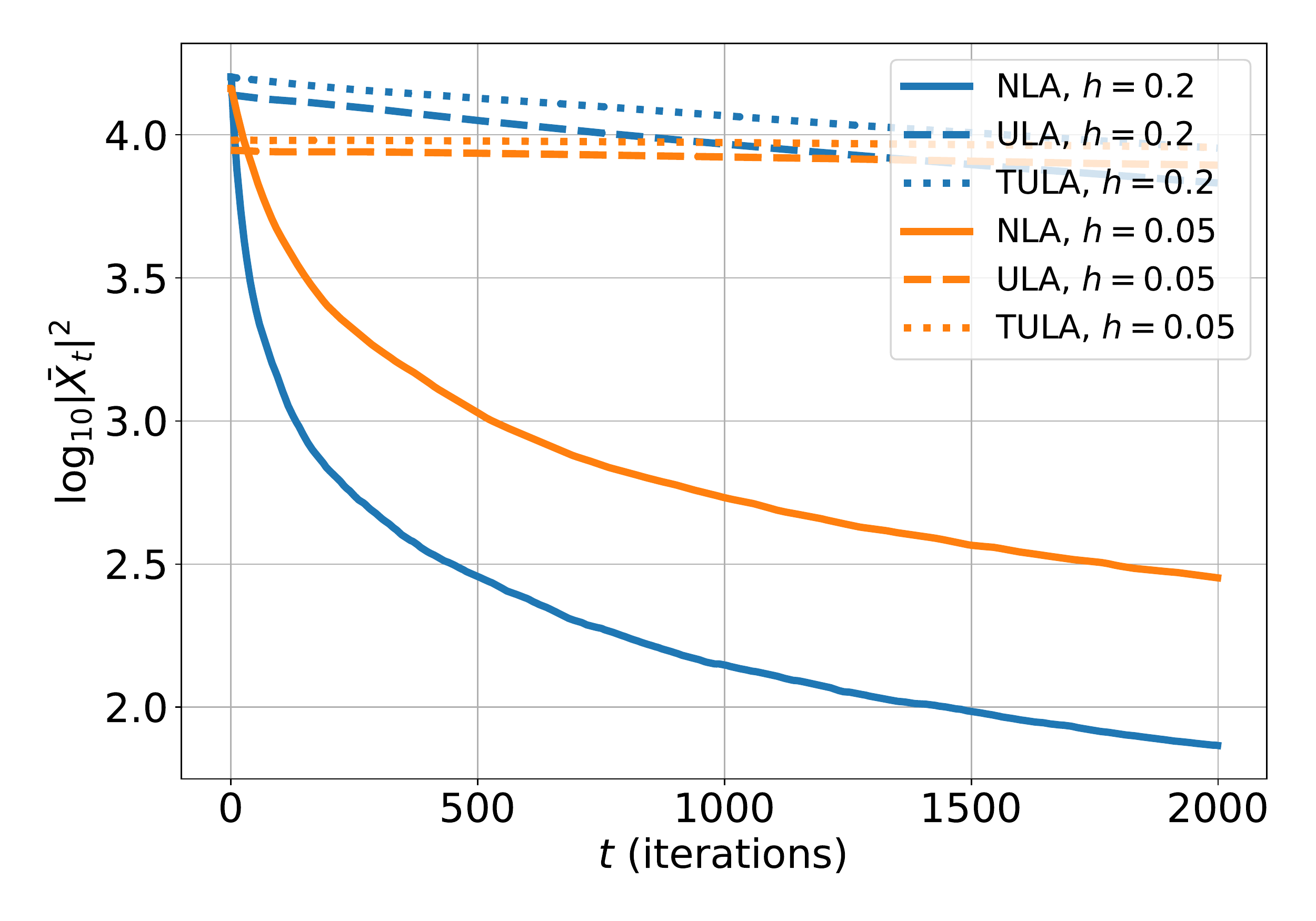}
    \includegraphics[width = 0.49\textwidth]{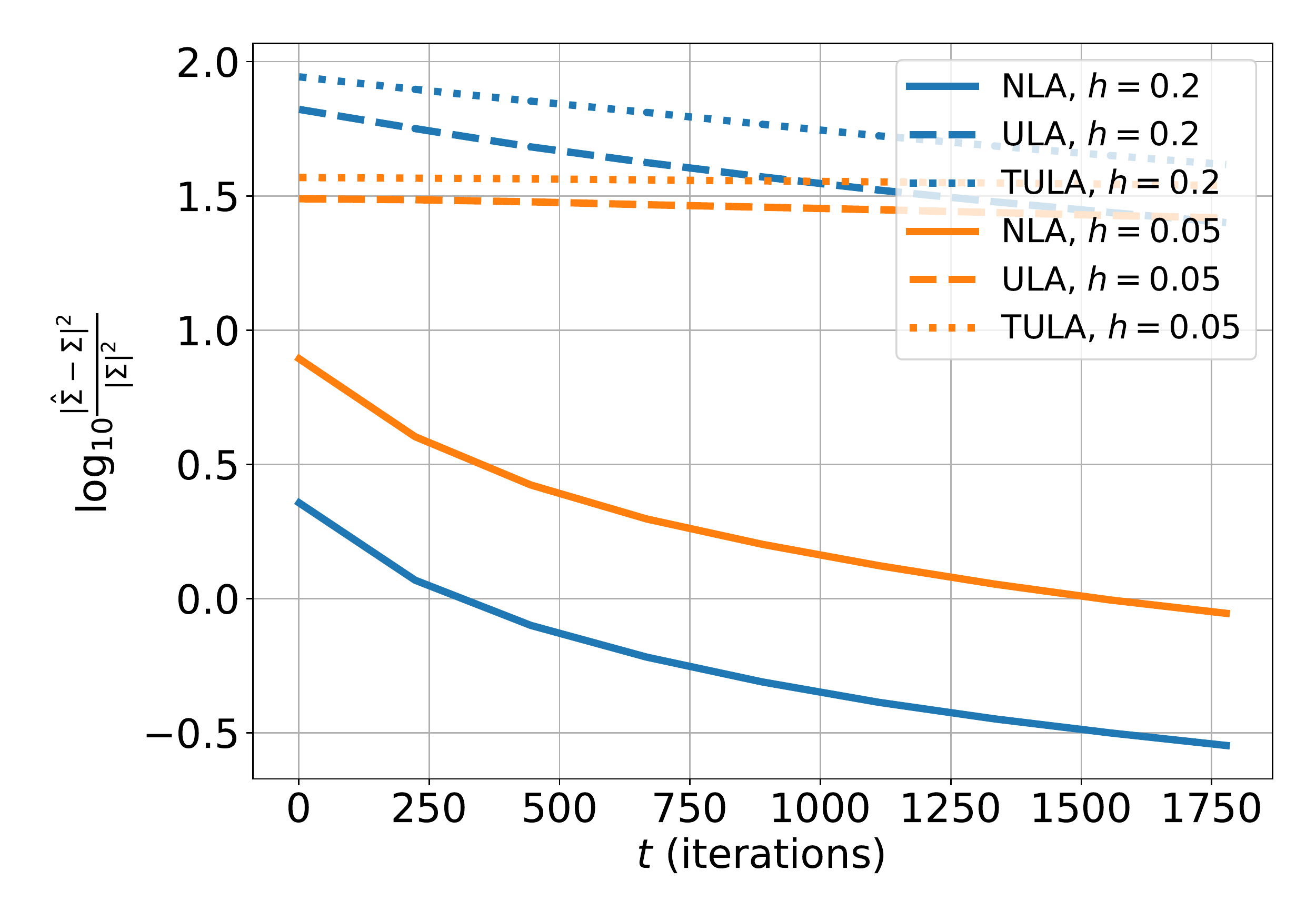}
    \caption{$V(x) = \langle x, \Sigma^{-1} x \rangle^{\sfrac{3}{4}}/2$, $\Sigma = \mathrm{diag}(1, 2,\dots, 100)$. Left: absolute squared error of the mean 0. Right: relative squared error for the scatter matrix $\Sigma$.
    }
    \label{fig:gengauss_mean}
\end{figure}
Figure~\ref{fig:gengauss_mean} compares the performance of \ref{eq:NLA} to that of the Unadjusted Langevin Algorithm (ULA)~\cite{durmus2019high} and of the Tamed Unadjusted Langevin Algorithm (TULA)~\cite{brosse2019tamed}.
We run the algorithms 50 times and compute running estimates for the mean and scatter matrix of the family following~\cite{zhang2013multivariate}.
Convergence is measured in terms of squared distance between means and relative squared distance between scatter matrices, $\| \hat \Sigma - \Sigma \|^2 / \| \Sigma \|^2$. 
\ref{eq:NLA} generates samples that rapidly approximate the true distribution and also displays stability to the choice of the step size.

\section{Open questions}
\label{sec:open}

We conclude this paper by discussing several intriguing directions for future research. In this paper, we focused on giving clean convergence results for the continuous-time diffusions \ref{eq:mld} and \ref{eq:newton}, and we leave open the problem of obtaining discretization error bounds. In discrete time, Newton's method can be unstable, and one uses methods such as damped Newton, Levenburg-Marquardt, or cubic-regularized Newton~\cite{conn2000trust,nesterov2006cubic}; it is an interesting question to develop sampling analogues of these optimization methods. In a different direction, we ask the following question: are there appropriate variants of other popular sampling methods, such as accelerated Langevin~\cite{ma2019there} or Hamiltonian Monte Carlo~\cite{neal2012mcmc}, which also enjoy the scale invariance of \ref{eq:newton}?

\medskip

\noindent\textbf{Acknowledgments}. \\
Philippe Rigollet was supported by NSF awards IIS-1838071, DMS-1712596, DMS-TRIPODS-1740751, and ONR
grant N00014-17- 1-2147.
Sinho Chewi and Austin J.\ Stromme were supported by the Department of
Defense (DoD) through the National Defense Science \& Engineering Graduate Fellowship (NDSEG)
Program.
Thibaut Le Gouic was supported by ONR grant N00014-17-1-2147 and NSF IIS-1838071.
\appendix

\section{Proof of the main convergence result}\label{appendix:markov_semigroup}


The law ${(\mu_t)}_{t\ge 0}$ of~\ref{eq:mld} satisfies the Fokker-Planck equation
\begin{align}\label{eq:fokker_planck_mld}
    \partial_t \mu_t
    &= \divergence\bigl(\mu_t \, {(\nabla^2 \phi)}^{-1} \nabla \ln \frac{\mu_t}{\pi}\bigr).
\end{align}
A unique solution to this equation, with enough regularity to justify our computations below, exists under fairly benign conditions on $\phi$ and $V$, see~\cite[Proposition 6]{lebris2008fokkerplanck}.

As discussed in Section~\ref{sec:mld}, it suffices to prove the convergence result in chi-squared divergence. The convergence results for total variation distance, Hellinger distance, and KL divergence follow from the inequalities~\cite[\S 2.4]{tsybakov2009nonparametric}
\begin{align*}
    2 \norm{\mu - \pi}_{\rm TV}^2,\  H^2(\mu,\pi),\ D_{\rm KL}(\mu\mmid \pi)\le \chi^2(\mu\mmid \pi), \qquad \forall \mu \ll \pi,
\end{align*}
while the convergence in Wasserstein distance follows from~\eqref{eq:liu_transport}.

\begin{proof}[Proof of Theorem~\ref{thm:exp_chi_2}]
    Using the Fokker-Planck equation~\eqref{eq:fokker_planck_mld}, we may compute
    \begin{align*}
        \partial_t \chi^2(\mu_t \mmid \pi)
        &= \partial_t \int \frac{\mu_t^2}{\pi}
        = 2\int \frac{\mu_t}{\pi} \, \partial_t \mu_t
        = 2\int \frac{\mu_t}{\pi} \, \divergence\bigl(\mu_t \, {(\nabla^2 \phi)}^{-1} \nabla \ln \frac{\mu_t}{\pi}\bigr) \\
        &= -2 \int \bigl\langle \nabla \frac{\mu_t}{\pi}, {(\nabla^2 \phi)}^{-1} \nabla \ln \frac{\mu_t}{\pi} \bigr\rangle \, \mu_t
        = -2 \int \bigl\langle \nabla \frac{\mu_t}{\pi}, {(\nabla^2 \phi)}^{-1} \nabla \frac{\mu_t}{\pi} \bigr\rangle \, \pi.
    \end{align*}
    The mirror Poincar\'e inequality~\eqref{eq:mp} implies that this quantity is at most $-2C_{\msf{MP}}^{-1} \chi^2(\mu_t\mmid \pi)$, which completes the proof via Gr\"onwall's inequality.
\end{proof}

We may reinterpret this proof within Markov semigroup theory.

\begin{proof}[Proof of Theorem~\ref{thm:exp_chi_2} from a Markov semigroup perspective]
In this proof, we denote the semigroup of~\ref{eq:mld} by ${(P_t)}_{t\ge 0}$; we refer readers to~\cite{bakry2014markov, vanhandelProbabilityHighDimension2014} for background on Markov semigroup theory.
The Dirichlet form $\mc{E}$ is given by
\[
\mc{E}(f,g)=\int \langle \nabla f,{(\nabla^2\phi)}^{-1}\nabla g\rangle \, \D\pi.
\]
Since it is a self-adjoint semigroup,
we get for all $f \in L^2(\pi)$,
\[
\int P_t\bigl(\frac{\D \mu_0}{\D\pi}\bigr) f\, \D \pi=\int \bigl(\frac{\D\mu_0}{\D\pi}\bigr)P_t f \, \D \pi=\int P_t f \, \D \mu_0=\int  f \, \D \mu_t=\int  \frac{\D \mu_t}{\D\pi}f \,\D\pi\, \,,
\]
so that
\[
P_t\bigl(\frac{\mu_0}{\pi}\bigr)= \frac{\mu_t}{\pi}.
\]
Therefore,
\[
\chi^2(\mu_t\mmid \pi):=\var_\pi\bigl(\frac{\D \mu_t}{\D \pi} \big)=\var_\pi P_t\bigl(\frac{\D \mu_0}{\D \pi}\bigr).
\]
Them, using a classical result of Markov semigroup theory (see for instance~\cite[Theorem 2.1]{cattiaux2009trendtv} or~\cite[Theorem 4.2.5]{bakry2014markov}),
\[
\chi^2(\mu_t \mmid \pi)=\var_\pi P_t\bigl(\frac{\D\mu_0}{\D\pi}\bigr)\le \e^{-\frac{2t}{C}}\var_\pi \bigl(\frac{\D\mu_0}{\D\pi}\bigr)= \e^{-\frac{2t}{C}}\chi^2(\mu_0\mmid \pi)
\]
if and only if the semigroup ${(P_t)}_{t\ge 0}$ satisfies
\begin{equation}\label{eq:poincaresemigroup}
\var_\pi(f)\le C\mathcal{E}(g,g) , \qquad \text{for all }~g \in D(\mc{E}),
\end{equation}
where $\mathcal{E}$ is the Dirichlet form of ${(P_t)}_{t\ge 0}$ with domain $D(\mc{E})$.
To conclude the proof, it suffices to note that \eqref{eq:poincaresemigroup} is precisely our assumption \eqref{eq:mp} with $C=C_{\msf{MP}}$.
\end{proof}


\section{Convergence in 2-Wasserstein distance}\label{scn:wasserstein}

\subsection{Background}

As we have discussed, the proof of Theorem~\ref{thm:exp_chi_2} in Appendix~\ref{appendix:markov_semigroup} implies that for any strictly log-concave target measure, the Newton-Langevin diffusion converges exponentially fast in the following error metrics: chi-squared divergence, KL divergence, Hellinger distance, and total variation distance. We also remark that convergence in R\'enyi divergences can also be proved in this setting, as in~\cite{vempala2019langevin}. On the other hand, we would also like to know if we can obtain convergence results for \emph{optimal transport} distances \cite{villani2003topics}. As a first step, the transportation inequality of~\cite{cordero2017transport},
\begin{align*}
    D_{\rm KL}(\mu \mmid \pi)
    &\ge \mc T_{D_V}(\mu\mmid\pi)
    := \inf\{\E D_V(X\mmid Z) : (X,Z)~\text{is a coupling of}~(\mu,\pi)\},
\end{align*}
which holds for all $\mu \ll \pi$, implies exponentially fast convergence in the asymmetric transportation cost $\mc T_{D_V}$, where $D_V(\cdot\mmid\cdot)$ is the Bregman divergence associated with $V$.

We turn towards the question of convergence in the $2$-Wasserstein distance (denoted $W_2$).
When $\pi$ is strongly log-concave, there is an elegant direct proof of exponential contraction in $W_2$ via a coupling of the Langevin process (see~\cite[Exercise 9.10]{villani2003topics}).
In general, however, convergence in $W_2$ is typically deduced from convergence in KL divergence, with the help of a \emph{transportation-cost inequality}
\begin{align}\label{eq:transportation}
    W_2^2(\mu,\pi)
    &\le C D_{\rm KL}(\mu\mmid \pi).
\end{align}
It has been known since the work of~\cite{otto2000generalization} that a log-Sobolev inequality~\eqref{eq:lsi} with constant $C_{\msf{LSI}}$ implies the validity of~\eqref{eq:transportation} with constant $C = C_{\msf{LSI}}$.
Since an LSI may not always hold or may hold with a poor constant,~\cite{bolleyWeightedCsiszarKullbackPinskerInequalities2005} provides weaker conditions: namely, if there exists $\alpha > 0$ such that
\begin{align}\label{eq:sq_exp_moment}
    \int \exp(\alpha \norm{x-x_0}^2) \, \D \pi(x) < \infty,
\end{align}
then we have the weaker inequality
\begin{align*}
    W_2^2(\mu,\pi)
    &\lesssim D_{\rm KL}(\mu \mmid \pi) + \sqrt{D_{\rm KL}(\mu \mmid \pi)}.
\end{align*}
Therefore, either the validity of an LSI or a square exponential moment suffice to transfer convergence in KL divergence to convergence in $W_2$.
In fact, it turns out that the log-Sobolev inequality~\eqref{eq:lsi}, the transportation inequality~\eqref{eq:transportation}, and the square exponential moment condition~\eqref{eq:sq_exp_moment} are all equivalent for log-concave measures,
and they are in general strictly stronger than the Poincar\'e inequality~\eqref{eq:p}~\cite{bobkovIsoperimetricAnalyticInequalities1999, otto2000generalization, bolleyWeightedCsiszarKullbackPinskerInequalities2005}.

Since Theorem~\ref{thm:exp_chi_2} provides a stronger control, namely in chi-squared divergence rather than in KL divergence, the reader might wonder if a weaker transportation inequality in which the RHS of~\eqref{eq:transportation} is replaced by $C{\chi^2(\mu \mmid \pi)}^{1/p}$ might hold under weaker assumptions. Indeed, the recent works~\cite{ding2015quadratictransport, ledoux2018remarks, liu2020quadratictransport} answer this question positively by showing that the Poincar\'e inequality~\eqref{eq:p} implies the transportation-cost inequality
\begin{align}\label{eq:liu_transport_p}
    W_2^2(\mu,\nu)
    &\le C \inf_{p\ge 1}\{p^2 {\chi^2(\mu\mmid \pi)}^{1/p}\}, \qquad\forall \mu \ll \pi
\end{align}
with constant $C = 2C_{\msf P}$.
In fact, the converse also holds: the validity of~\eqref{eq:liu_transport_p} implies the Poincar\'e inequality~\eqref{eq:p} with constant $C_{\msf P} = C/\sqrt 2$.

If we specialize this result to the case $p=2$, then the Poincar\'e inequality~\eqref{eq:p} implies
\begin{align}\label{eq:liu_transport_1/2}
    W_2^2(\mu,\nu)
    \le 8 C_{\msf P} \sqrt{\chi^2(\mu \mmid \pi)}, \qquad \forall \mu \ll \pi.
\end{align}
In the next section, we give a proof of the inequality~\eqref{eq:liu_transport_1/2} with a slightly worse constant, i.e., with $9$ instead of $8$.


We now briefly describe the method of~\cite{otto2000generalization}, since it is relevant for our approach.
Otto and Villani work in the framework of \emph{Otto calculus}, which interprets~\ref{eq:langevin_sde} as the {gradient flow of the KL divergence in the space of probability measures equipped with the $W_2$ metric}.
As discussed in Section~\ref{sec:mld}, an~\ref{eq:lsi} is a \ref{eq:pl} inequality, which ensures rapid convergence of the gradient flow.
This is then used to deduce the transportation-cost inequality~\eqref{eq:transportation}.

We follow the argument of Otto and Villani, but consider the \emph{gradient flow of the chi-squared divergence} instead of the KL divergence.
We prove
a \L{}ojasiewicz inequality for the chi-squared divergence, and use the gradient flow to deduce~\eqref{eq:liu_transport_1/2} (with a slightly worse constant).

\subsection{Proof of the chi-squared transportation inequality}

Following the proof outline above, we start by proving a PL-type inequality for the chi-squared divergence. Using tools developed in~\cite{ambrosio2008gradient}, it is a standard exercise to establish that the Wasserstein gradient of the functional $\mu \mapsto \chi^2(\mu\mmid \pi)$ is given by $2\nabla(\D \mu/\D \pi)$. Therefore, the right-hand side of the following inequality involves the squared norm of the Wasserstein gradient of the chi-squared divergence, where we use the norm corresponding to the Riemannian structure of Wasserstein space (see~\cite[\S 8]{ambrosio2008gradient}). Note that since the objective is raised to the power $3/2$ on the left-hand side it is not quite a \ref{eq:pl} inequality, and rather it is a form commonly referred to as a {\L}ojasiewicz inequality~\cite{lojasiewicz1963propriete} with parameter $3/4$. 

\begin{prop}
Let $C_{\msf P} \in (0, \infty]$ denote the Poincar\'e constant of $\pi$. Then,
    \begin{align*}
      {\chi^2(\mu \mmid \pi)}^{3/2} \le    \frac{9C_{\msf P}}{4} \int \bigl\lVert \nabla \frac{\D\mu}{\D\pi}\bigr\rVert^2 \, \D \mu, \qquad \forall \mu \ll \pi.
    \end{align*}
\end{prop}
\begin{proof}
Using the Poincar\'e inequality~\eqref{eq:p},
we obtain
\begin{align*}
    \int \bigl\lVert \nabla \frac{\D\mu}{\D\pi} \bigr\rVert^2 \, \D \mu
    &= \int \bigl\lVert \nabla \frac{\D\mu}{\D\pi} \bigr\rVert^2 \, \frac{\D\mu}{\D \pi} \, \D \pi
    = \frac{4}{9} \int \bigl\lVert \nabla \bigl( \frac{\D\mu}{\D\pi} \bigr)^{3/2} \bigr\rVert^2 \, \D \pi
    \ge \frac{4}{9C_{\msf P}} \var_\pi\bigl(\bigl( \frac{\D\mu}{\D\pi} \bigr)^{3/2}\bigr).
\end{align*}
In the following steps, we apply the following: (1) $\var X \le \E[\abs{X-c}^2]$ for any $c \in \R$; (2) $x\mapsto x^{2/3}$ is $2/3$-H\"older continuous with unit constant; (3) Jensen's inequality.
\begin{align*}
    \chi^2(\mu \mmid \pi)
    = \var_\pi \bigl(\frac{\D\mu}{\D\pi}\bigr)
    &\overset{(1)}{\le} \E_\pi\Bigl[\Bigl\lvert \frac{\D\mu}{\D\pi} - \E_\pi\bigl[\bigl( \frac{\D\mu}{\D\pi} \bigr)^{3/2}\bigr]^{2/3} \Bigr\rvert^2\Bigr]\\
    &\overset{(2)}{\le} \E_\pi\Bigl[\Bigl\lvert\bigl( \frac{\D\mu}{\D\pi} \bigr)^{3/2} - \E_\pi\bigl[\bigl( \frac{\D\mu}{\D\pi} \bigr)^{3/2} \bigr] \Bigr\rvert^{4/3}\Bigr] \\
    &\overset{(3)}{\le} {\E_\pi\Bigl[\Bigl\lvert\bigl( \frac{\D\mu}{\D\pi} \bigr)^{3/2} - \E_\pi\bigl[\bigl( \frac{\D\mu}{\D\pi} \bigr)^{3/2} \bigr] \Bigr\rvert^2\Bigr]}^{2/3}
    = \Bigl(\var_\pi\bigl( \bigl(\frac{\D\mu}{\D\pi} \bigr)^{3/2}\bigr)\Bigr)^{2/3}.
\end{align*}
This proves the result.
\end{proof}

\begin{thm}\label{thm:Poincare2chi2transport}
Suppose ${\chi^2(\cdot \mmid \pi)}$ satisfies the following {\L}ojasiewicz
inequality:
\begin{equation}\label{eq:hypqPL}
    {\chi^2(\mu\mmid \pi)}^{2/q}\le 4C_{\msf{PL}} \E_\mu\bigl[\bigl\lVert \nabla\frac{\D\mu}{\D\pi}\bigr\rVert^2\bigr], \qquad \forall \mu \ll \pi,
\end{equation}
for some $q \in (1,\infty)$.
Then, $\pi$ satisfies the chi-squared transportation inequality
\[
    W_2^2(\mu,\pi)\le p^2 C_{\msf{PL}} \, {\chi^2(\mu \mmid \pi)}^{2/p}, \qquad \forall \mu \ll \pi,
\]
where $1/p + 1/q = 1$.
\end{thm}

\begin{proof}
The proof follows~\cite{otto2000generalization}.
Take a path ${(\mu_t)}_{t\ge 0}$ starting at some $\mu_0=\mu$ and following the $W_2$ gradient flow of the chi-squared divergence $\chi^2(\cdot\mmid \pi)$, that is,
\begin{align*}
    \partial_t \mu_t
    &= 2 \divergence\bigl(\mu_t \nabla \frac{\mu_t}{\pi}\bigr).
\end{align*}
The existence of this gradient flow and the regularity required for the following computations can be justified by~\cite{ohta2011generalizedentropies, ohta2013generalizedentropiesii} and~\cite[Theorem 11.2.1]{ambrosio2008gradient}.
Denote by $T_t$ the optimal transport map sending $\mu_t$ to $\mu_0$. Then, the time derivative of the squared Wasserstein distance can be computed as in~\cite[Corollary 10.2.7]{ambrosio2008gradient} to be
\begin{align*}
    \partial_t W_2^2(\mu_0,\mu_t)
    &= - 4\E_{\mu_t}\bigl\langle \nabla\frac{\mu_t}{\pi},T_t-\id\bigr\rangle
    \le 4 W_2(\mu_0,\mu_t) \E_{\mu_t} \bigr\lVert \nabla\frac{\mu_t}{\pi}\bigr\rVert\,,
\end{align*}
where we apply the Cauchy-Schwarz and Jensen inequalities.
It yields
\[
\partial_t W_2(\mu_0,\mu_t)\le 2 \E_{\mu_t} \bigl\lVert \nabla\frac{\mu_t}{\pi} \bigr\rVert.
\]
Also, the chi-squared divergence satisfies
\begin{align*}
\partial_t\chi^2(\mu_t \mmid \pi)&=-4\E_{\mu_t}\bigl[\bigr\lVert \nabla\frac{\mu_t}{\pi}\bigr\rVert^2\bigr].
\end{align*}
Using the assumption \eqref{eq:hypqPL},
\begin{align*}
\partial_t [{\chi^2(\mu_t \mmid \pi)}^{1/p}]
&=\frac{\partial_t\chi^2(\mu_t \mmid \pi)}{p {\chi^2(\mu_t \mmid \pi)}^{1/q}}
= - \frac{4}{p{\chi^2(\mu_t \mmid \pi)}^{1/q}} \E_{\mu_t}\bigl[\bigr\lVert \nabla\frac{\mu_t}{\pi}\bigr\rVert^2\bigr]
\le - \frac{2}{p \sqrt{C_{\msf{PL}}}} \E_{\mu_t}\bigl\lVert \nabla \frac{\mu_t}{\pi}\bigr\rVert.
\end{align*}
If we define 
\[
g(t):=W_2(\mu_0,\mu_t)+ p\sqrt{C_{\msf{PL}}} \, {\chi^2(\mu_t \mmid \pi)}^{1/p},
\]
we have proved that 
\[
g' \le 0.
\]
Since $g(0)= p\sqrt{C_{\msf{PL}}} \, {\chi^2(\mu_0 \mmid \pi)}^{1/p}$ and $\lim_{t\to\infty}g(t)=W_{2}(\mu,\pi)$, we have shown a transport inequality
\begin{align*}
W_2^2(\mu,\pi)
&\le p^2 C_{\msf{PL}} \, {\chi^2(\mu \mmid \pi)}^{2/p}.
\end{align*}
\end{proof}

\begin{thm}
Let $\pi$ be a distribution on $\R^d$ with finite Poincar\'e constant $C_{\msf{P}} > 0$. Then for any measure $\mu \in \mathcal{P}_2(\R^d)$, it holds
$$
 W_2^2(\mu,\pi) \le 9C_{\msf P} \sqrt{\chi^2(\mu \mmid \pi)} \,.
$$
\end{thm}
\begin{proof}
    The inequality follows immediately from the two preceding results.
\end{proof}

\begin{rmk}
    Transportation-cost inequalities for R\'enyi divergences were also studied in~\cite{ding2014divergencetransport, bobkovding2015renyitransport}.
\end{rmk}

\section{Additional choices for the mirror map}
\label{app:ext}

We extend our results to other choices of the mirror map $\phi$ that serve as proxies for $V$ and that also lead to exponential convergence of~\ref{eq:mld}. 

The first result below is useful in situations when there exists a strictly convex mirror map $\phi$ such $\nabla \phi$ is easier to invert than $\nabla V$. It ensures exponential ergodicity of \eqref{eq:mld} when $\nabla^2 V$ dominates $\nabla^2 \phi$ in the sense of the Loewner order.

\begin{cor}\label{cor:potential_dominates_mirror}
    Suppose that $\pi$ is strictly log-concave and that $\nabla^2 \phi \preceq C\nabla^2 V$, where $\preceq$ denotes the Loewner order. Then, the law ${(\mu_t)}_{t\ge 0}$ of~\ref{eq:mld} satisfies
    \begin{align*}
    2 \norm{\mu_t - \pi}_{\rm TV}^2,\ H^2(\mu_t,\pi),\ D_{\rm KL}(\mu_t \mmid \pi),\ \chi^2(\mu_t\mmid \mu),\ \frac{1}{2 C_{\msf P}} W_2^2(\mu_t,\pi) \le \e^{-\frac{2t}{C}} \chi^2(\mu_0 \mmid \pi).
    \end{align*}
\end{cor}
\begin{proof}
    The assumption implies
    \begin{align*}
        C\E_\pi\langle \nabla f, {(\nabla^2 \phi)}^{-1} \nabla f \rangle
        &\ge \E_\pi\langle \nabla f, {(\nabla^2 V)}^{-1} \nabla f \rangle
        \ge \var_\pi f,
    \end{align*}
    where again we apply the Brascamp-Lieb inequality.
    This verifies~\eqref{eq:mp} with constant $C_{\msf{MP}} = C$.
\end{proof}

Our second result does not require $\pi$ to be log-concave but only that it is close to a strictly log-concave distribution $\widetilde \pi$ in the following sense: the density of $\pi$ with respect to $\widetilde \pi$ is uniformly bounded away from $0$ and $\infty$.

\begin{cor}\label{cor:bounded_perturbation}
    Suppose that $\widetilde \pi = \exp(-\widetilde V)$ is strictly log-concave and suppose that $\pi$ has density $\rho$ w.r.t.\ $\widetilde \pi$. Let $M := (\sup \rho)/(\inf \rho)$. Then, the law ${(\mu_t)}_{t\ge 0}$ of~\ref{eq:mld} with mirror map $\phi = \widetilde V$ and target density $\pi$ satisfies
    \begin{align*}
     2 \norm{\mu_t - \pi}_{\rm TV}^2,\ H^2(\mu_t,\pi),\ D_{\rm KL}(\mu_t \mmid \pi),\ \chi^2(\mu_t\mmid \mu),\ \frac{1}{2C_{\msf P} M} W_2^2(\mu_t,\pi) \le \e^{-\frac{2t}{M}} \chi^2(\mu_0 \mmid \pi),
    \end{align*}
    where $C_{\msf P}$ is the Poincar\'e constant of $\widetilde \pi$.
\end{cor}
\begin{proof}
    It is standard that the Poincar\'e inequality~\eqref{eq:p}, and the mirror Poincar\'e inequality~\eqref{eq:mp}, are stable under bounded perturbations of the measure. It implies that $\pi$ satisfies a Poincar\'e inequality with constant $C_{\msf P} M$, and a mirror Poincar\'e inequality with constant $M$. We prove the latter statement for completeness; for the former statement, see~\cite[Problem 3.20]{vanhandelProbabilityHighDimension2014}.
    
    Observe that
    \begin{align*}
        \int \langle \nabla f, {(\nabla^2 \widetilde V)}^{-1} \nabla f \rangle \, \D\pi
        &= \int \langle \nabla f, {(\nabla^2 \widetilde V)}^{-1} \nabla f \rangle \, \frac{\D\pi}{\D\widetilde \pi} \, \D\widetilde \pi
        \ge (\inf \rho) \int \langle \nabla f, {(\nabla^2 \widetilde V)}^{-1} \nabla f \rangle \, \D\widetilde \pi
    \end{align*}
    and
    \begin{align*}
        \var_{\widetilde \pi} f
        &= \inf_{m\in\R^d} \int \|f - m\|^2 \, \D\widetilde \pi
        = \inf_{m\in\R^d} \int \|f - m\|^2 \, \frac{\D\widetilde \pi}{\D\pi} \, \D\pi \\
        &\ge \frac{1}{\sup \rho} \, \inf_{m\in\R^d} \int \|f-m\|^2 \, \D\pi
        = \frac{1}{\sup\rho} \var_\pi f.
    \end{align*}
    Combining these inqualities with the Brascamp-Lieb inequality for $\widetilde \pi$,
    \begin{align*}
        \int \langle \nabla f, {(\nabla^2 \widetilde V)}^{-1} \nabla f \rangle \, \D\widetilde \pi
        \ge \var_{\widetilde \pi} f,
    \end{align*}
    yields~\eqref{eq:mp} with constant $C_{\msf{MP}} = M$.
\end{proof}

\section{Stability in KL with respect to exp-concave perturbations}
\label{appendix:auxiliary}
The following lemma quantifies the approximation error of replacing $\pi$ by $\pi_\beta$ in Section~\ref{scn:unif_sampling} and, more generally provides a simple bound to control the KL divergence between a log-concave distribution and its perturbation by a $\nu$-exp-concave barrier function. Its proof uses crucially displacement convexity of the KL divergence to a log-concave measure~\cite[\S 5]{villani2003topics}, and it can be viewed as the sampling analogue of~\cite[(4.2.17)]{nesterov2004optimization}.

Recall that $b$ is \emph{$\nu$-exp-concave} if the mapping $\exp(-\nu^{-1} b)$ is concave.


\begin{lem}\label{lem:self_concordant}
Let $\pi$ be a log-concave distribution  on a convex set $\eu K \subset \R^d$. Fix $\nu>0$, and let $\widetilde \pi$  have density $\exp(- b)$ with respect to $\pi$, where  $b: \eu K \to \R$ is $\nu$-exp-concave. Then it holds that
$$
D_{\rm KL}(\widetilde \pi \mmid \pi) \le \nu\,.
$$

\end{lem}
\begin{proof}
    On $\interior \eu K$, we have
    \begin{align}
    \label{eq:pr:sc:1}
        -\nabla \ln \frac{\D\widetilde \pi}{\D \pi}= \nabla b.
    \end{align}
    The measure $\pi$ is log-concave, so by displacement convexity of entropy~\cite[Theorem 9.4.11]{ambrosio2008gradient} and the ``above-tangent'' formulation of convexity~\cite[Proposition 5.29]{villani2003topics}, we have
    \begin{align*}
        0
        &= D_{\rm KL}(\pi \mmid \pi)
        \ge D_{\rm KL}(\widetilde \pi \mmid \pi) + \E\bigl\langle \nabla \ln \frac{\D \widetilde \pi}{\D \pi}(\widetilde X), X - \widetilde X \bigr\rangle,
    \end{align*}
    where $(X,\widetilde X)$ are optimally coupled for $\pi$ and $\widetilde \pi$.
    If we rearrange this inequality and use the identities in~\eqref{eq:pr:sc:1}, we get
    \begin{align}
     \label{eq:pr:sc:2}
        D_{\rm KL}(\widetilde \pi \mmid \pi)
        &\le - \E\bigl\langle \nabla \ln \frac{\D\widetilde \pi}{\D\pi}(\widetilde X), X - \widetilde X \bigr\rangle
        = \E\langle \nabla b(\widetilde X), X - \widetilde X \rangle\,.
    \end{align}
    We now use the fact that $b$ is $\nu$-exp-concave. To that end, define the convex function 
$$
\varphi(t)=-\exp\bigl(-\frac1\nu b(\widetilde X+t\, (X-\widetilde X))\bigr), \qquad t \in [0,1]\,.
$$
By convexity, we have 
$$
\varphi'(0)\cdot(1-0) \le \varphi(1)-\varphi(0) \le -\varphi(0)=\exp\bigl( - \frac{1}{\nu} b(\widetilde X)\bigr).
$$
Since 
$$
\varphi'(0)= \frac{1}{\nu} \exp\bigl(- \frac{1}{\nu}b(\widetilde X)\bigr) \, \langle \nabla b(\widetilde X), X- \widetilde X\rangle\,,
$$
the above inequality reads $\langle \nabla b(\widetilde X), X - \widetilde X \rangle \le \nu$, which completes the proof together with~\eqref{eq:pr:sc:2}.
\end{proof}

\begin{rmk}
    It is known that given any convex body $\eu C \subset \R^d$, there exists a standard self-concordant $\nu^{-1}$-exp-concave barrier with $\nu \le d$~\cite{nesterov1995interiorpoint, bubeck2015entropic, yintatlee2018universalbarrier}.
\end{rmk}

\section{Numerical experiments}
\label{appendix:numericals}

In this section, we gather additional details and figures to support our numerical experiments. First, in Section~\ref{subsec:moregauss}, we display the samples from our Gaussian experiment. Then, Section~\ref{subsec:lrdetails} gives details of the Bayesian logistic regression experiment displayed in Figure~\ref{fig:logistic} and shows the effect of varying step size. Section~\ref{subsec:unif_cvx} gives details of sampling from an ill-conditioned convex set. Finally, Section~\ref{subsec:samplestrict} shows an experiment where we use the \ref{eq:NLA} and a Mirror-Langevin Algorithm~\ref{eq:mla} to approximately sample from a degenerate log-concave distribution.

\subsection{Sampling from a Gaussian distribution}
\label{subsec:moregauss}

We display some supplementary experiments for the elliptically symmetric scatter family example of Section~\ref{sec:sim}. First, we repeat the example in Figure~\ref{fig:gengauss_mean} for the simpler case of the Gaussian distribution ($\gamma=1$) on $\R^{100}$ with the same scatter matrix $\Sigma=\diag(1, 2, \dots, 100)$ in Figure~\ref{fig:gauss_mean}. We again see the superiority of \ref{eq:NLA} over the  Unadjusted Langevin Algorithm (ULA)~\cite{durmus2019high} and the Tamed Unadjusted Langevin Algorithm (TULA)~\cite{brosse2019tamed}. Here and in Section~\ref{sec:sim} the additional parameter of TULA (denoted $\gamma$ in \cite{brosse2019tamed}) is chosen equal to $.1$.

\begin{figure}[ht]
    \vspace{-0.0cm}
    \centering
    \includegraphics[width = 0.49\textwidth]{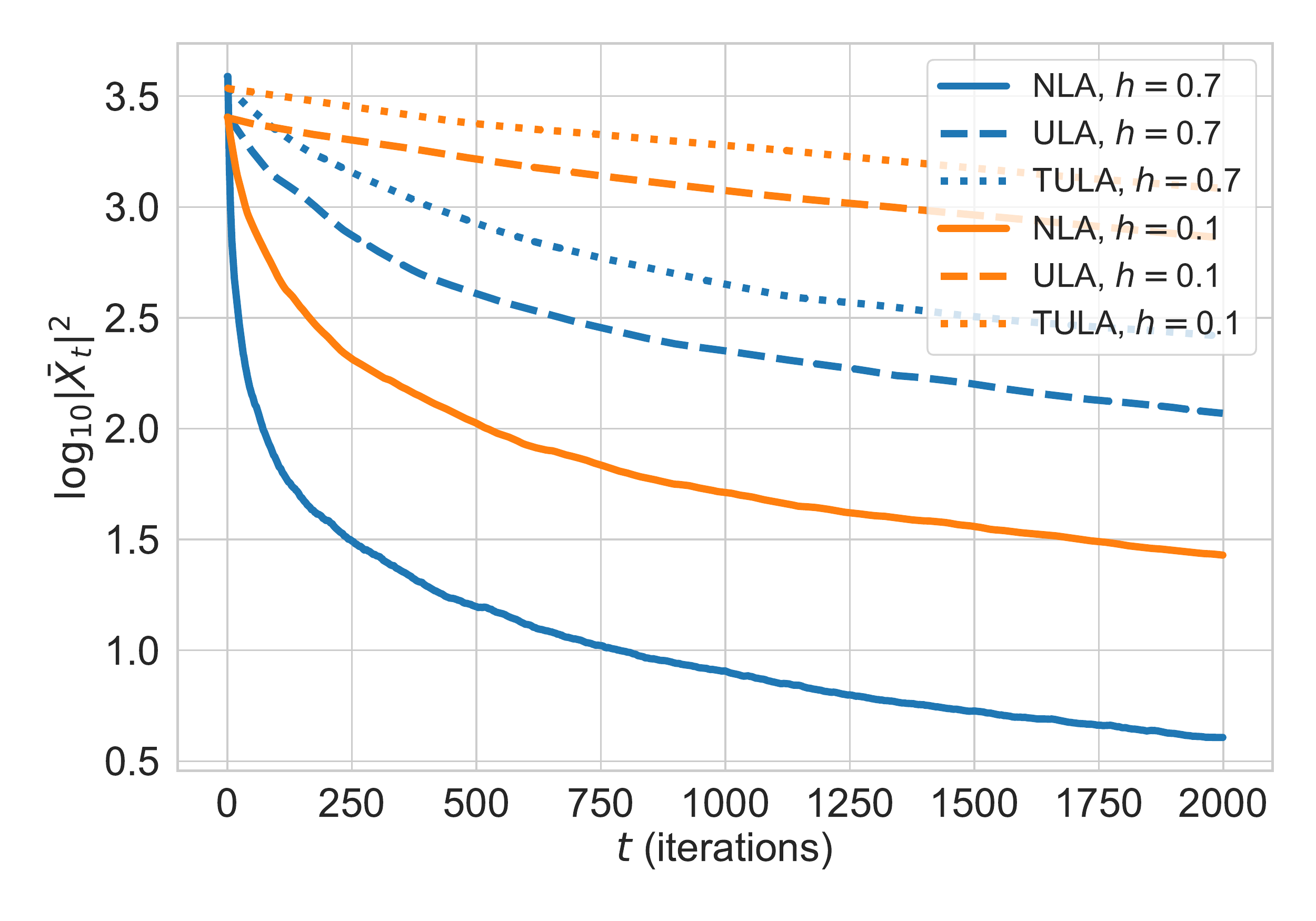}
    \includegraphics[width = 0.49\textwidth]{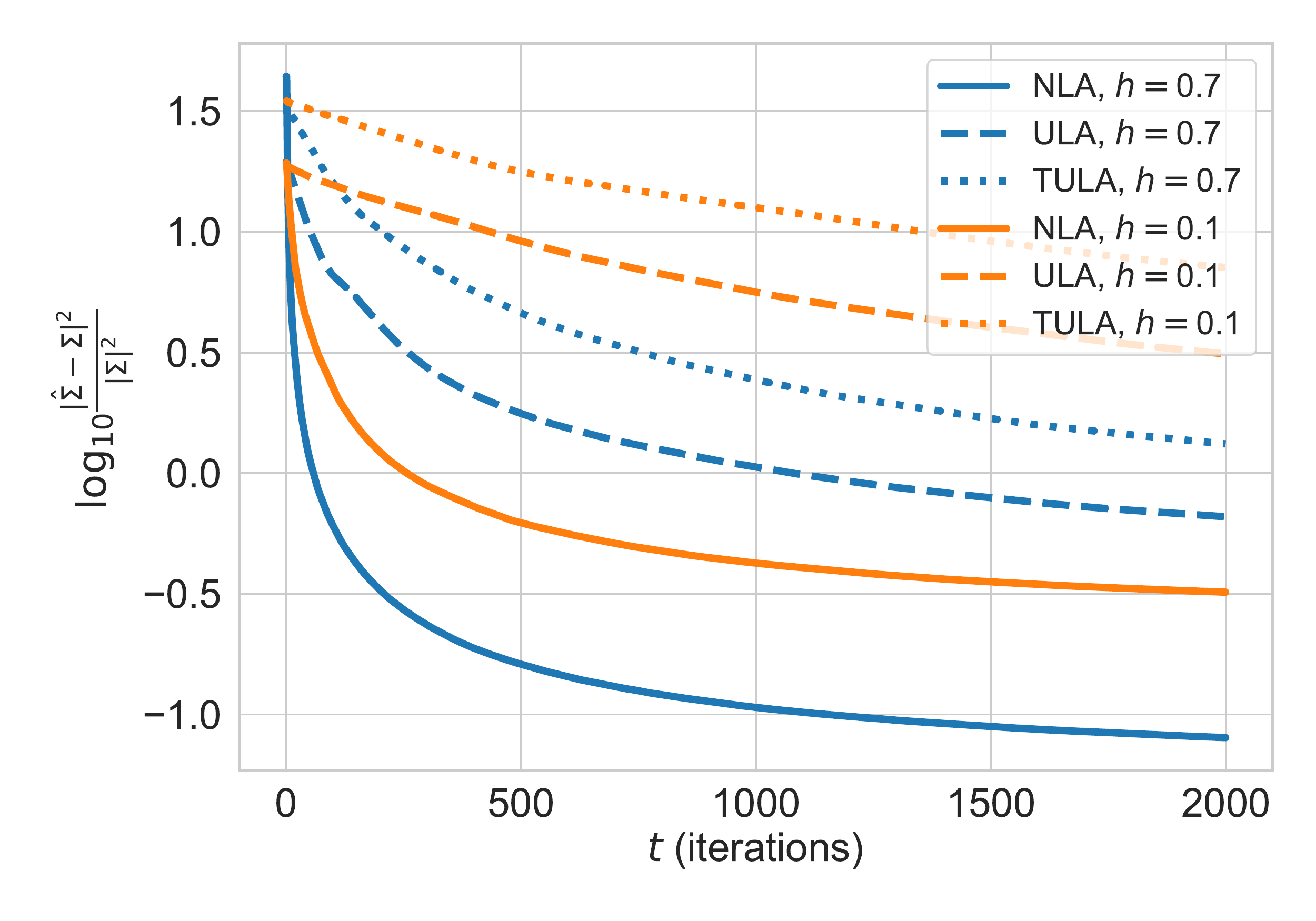}
    \caption{We display convergence of the various algorithms for an ill-conditioned Gaussian distribution, with $d=100$ and $\Sigma = \mathrm{diag}(1, 2,\dots, 100)$. Left: error is the squared distance from $0$. Right: error is the relative distance between scatter matrices. As in the experiment displayed in Figure~\ref{fig:gengauss_mean}, \ref{eq:NLA} rapidly converges both in terms of location and scale for large step sizes.}
    \label{fig:gauss_mean}
    \vspace{-0cm}
\end{figure}

We also display some samples from the Gaussian experiment of Figure~\ref{fig:gauss_mean} in Figure~\ref{fig:gauss_samples}. \ref{eq:NLA} maintains good performance for a wide range of step-size choices, while ULA and TULA require a small step size to accurately sample from the target distribution. In fact, even with a small step size, ULA and TULA often jump to small probability regions, while \ref{eq:NLA} avoids these regions even for large step sizes.

\begin{figure}[ht!]
    \vspace{-0cm}
    \begin{center}
    \includegraphics[width = 0.3\textwidth]{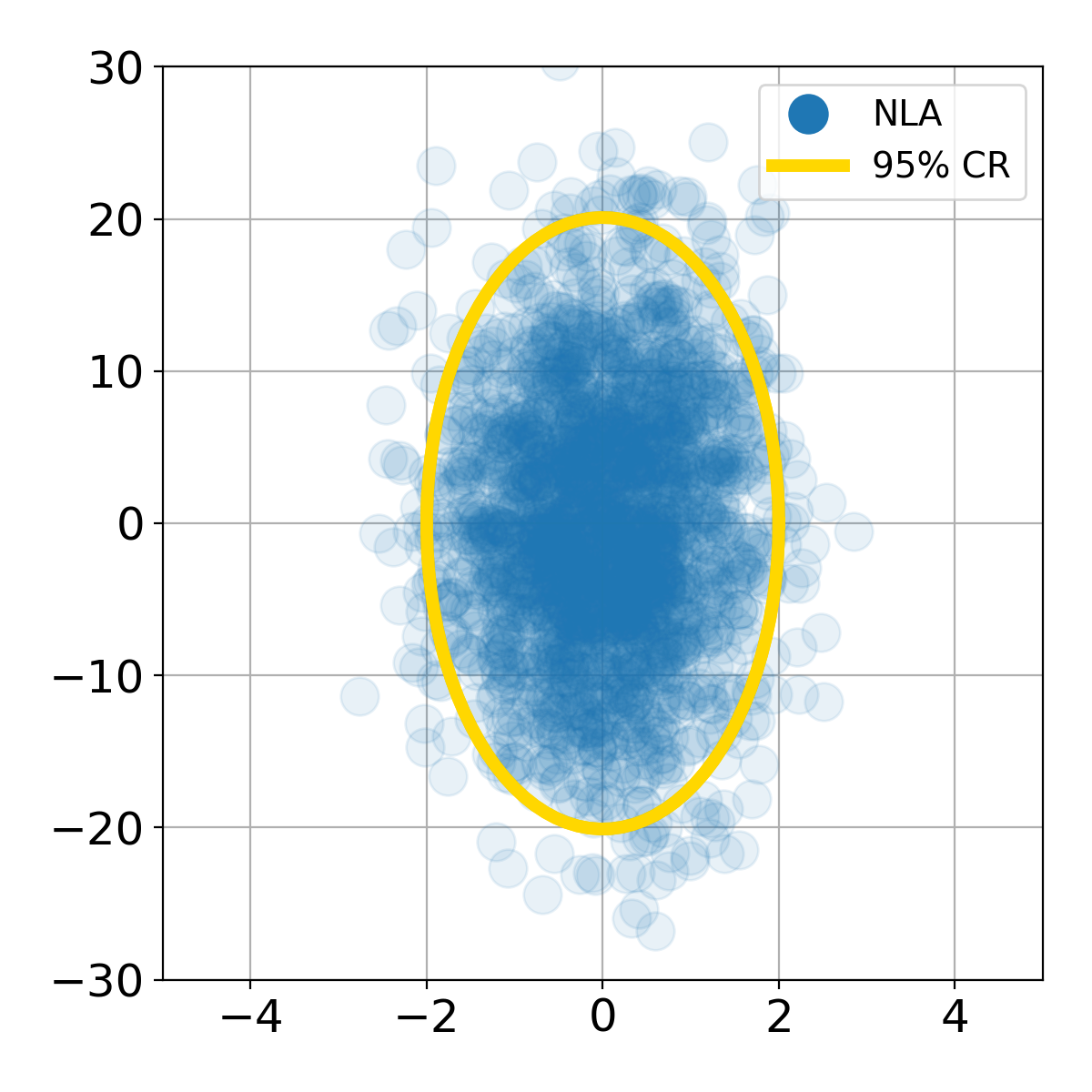}
    \includegraphics[width = 0.3\textwidth]{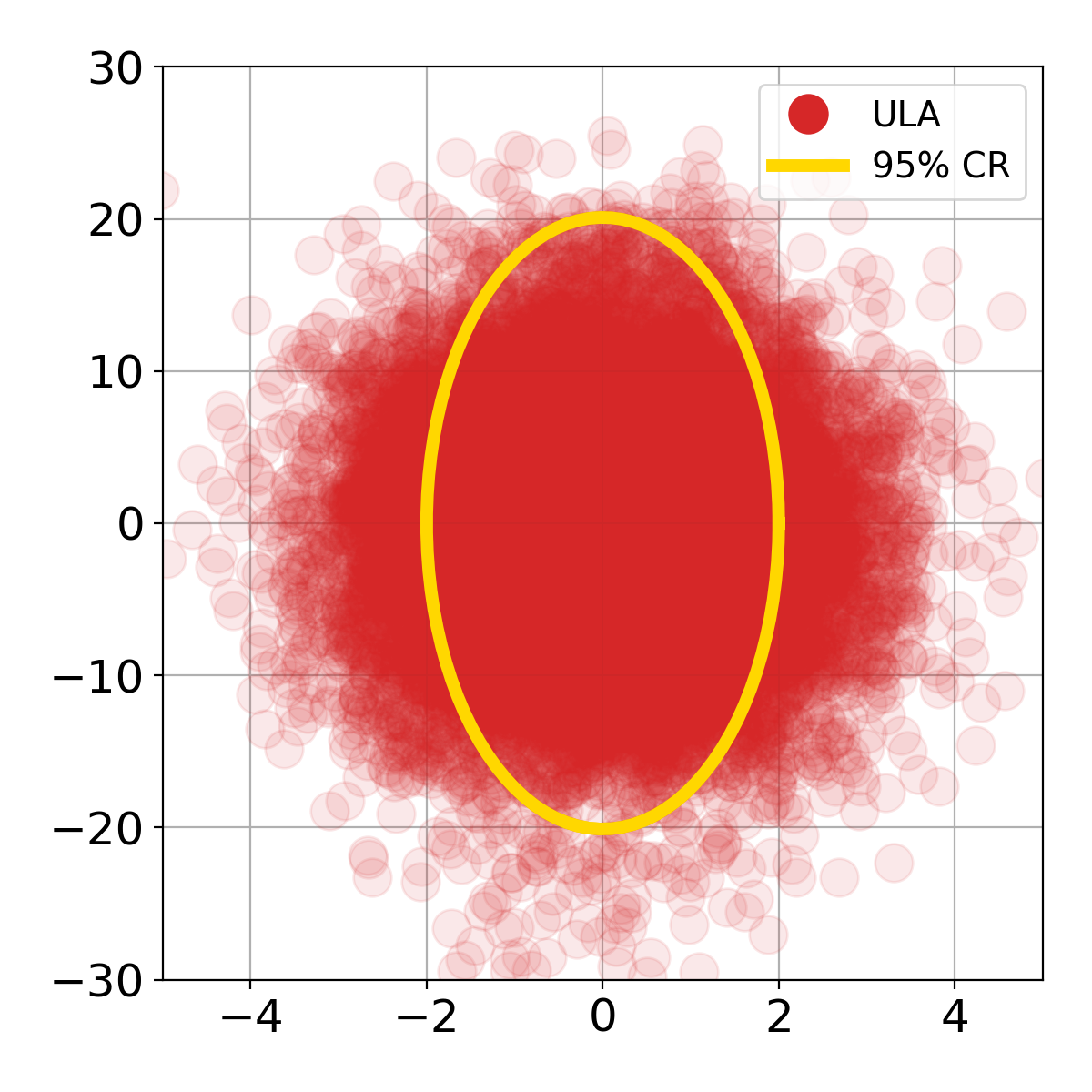}
    \includegraphics[width = 0.3\textwidth]{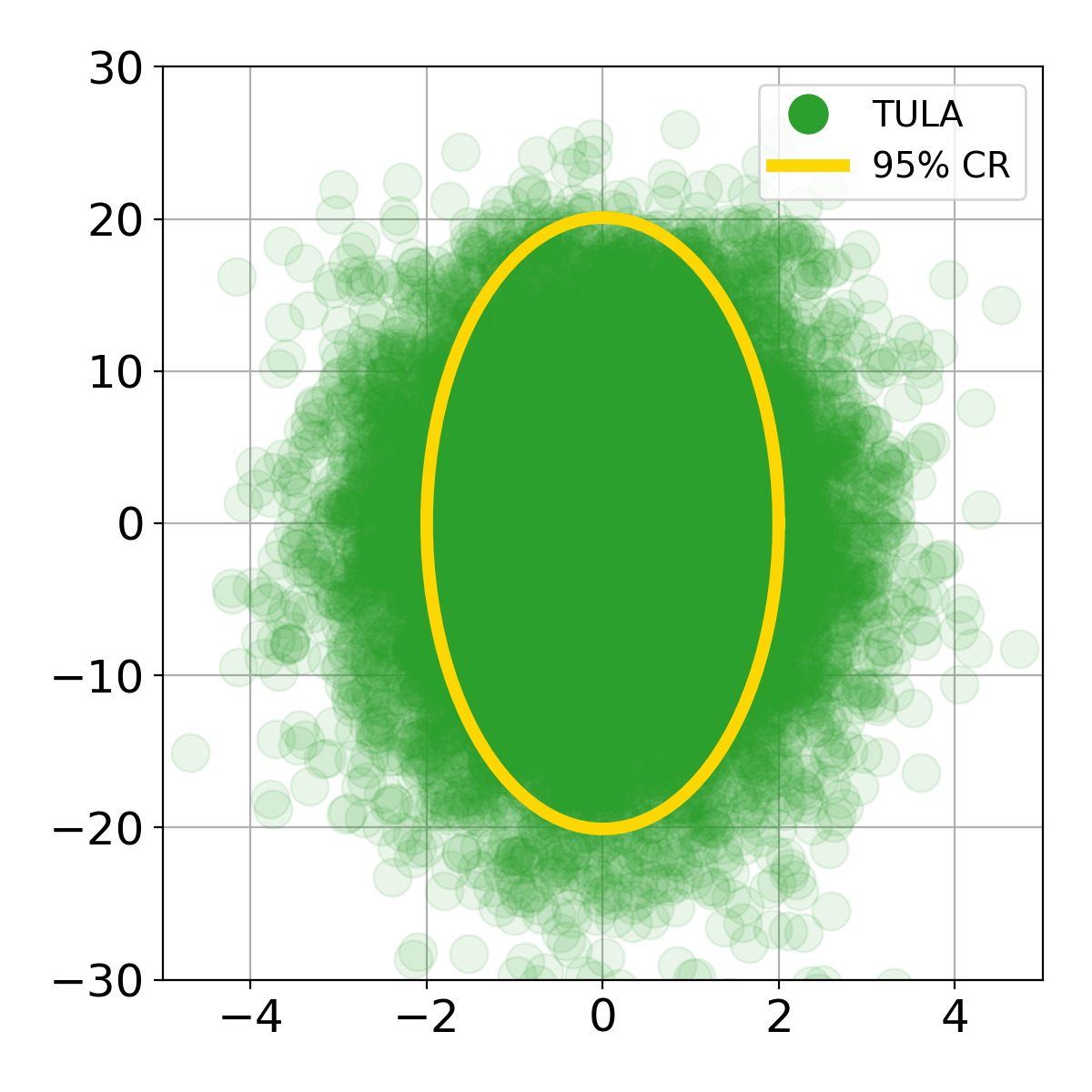}
    \includegraphics[width = 0.3\textwidth]{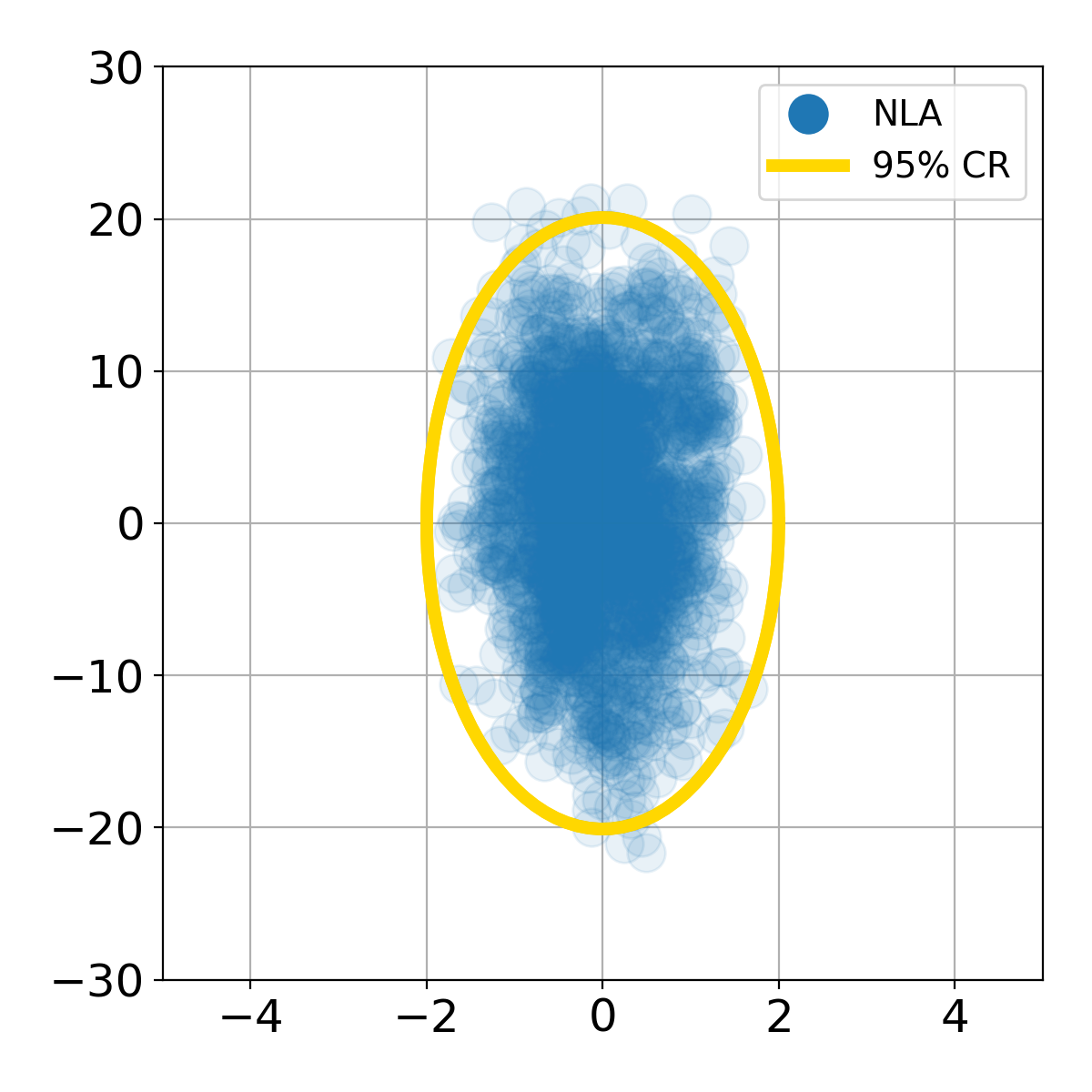}
    \includegraphics[width = 0.3\textwidth]{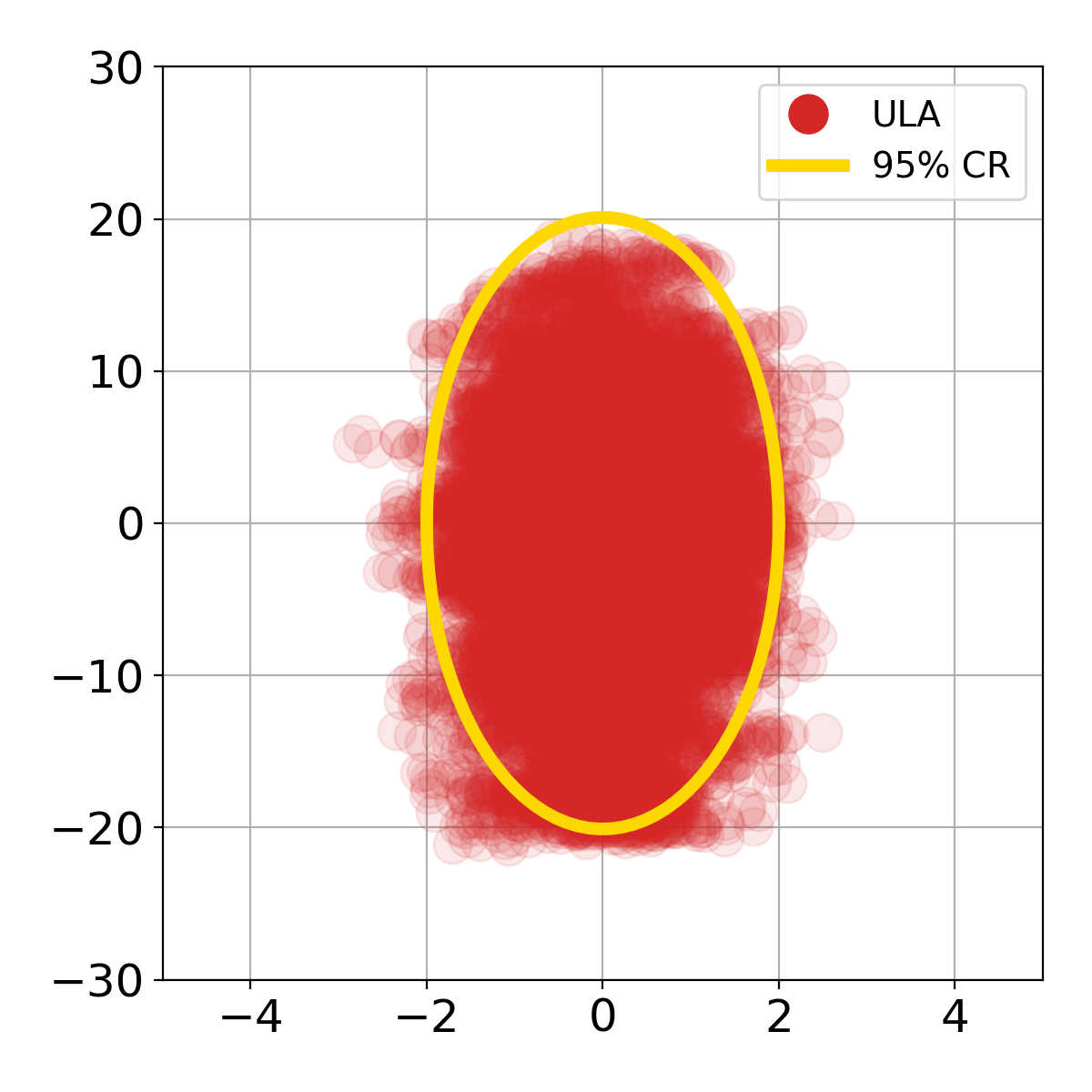}
    \includegraphics[width = 0.3\textwidth]{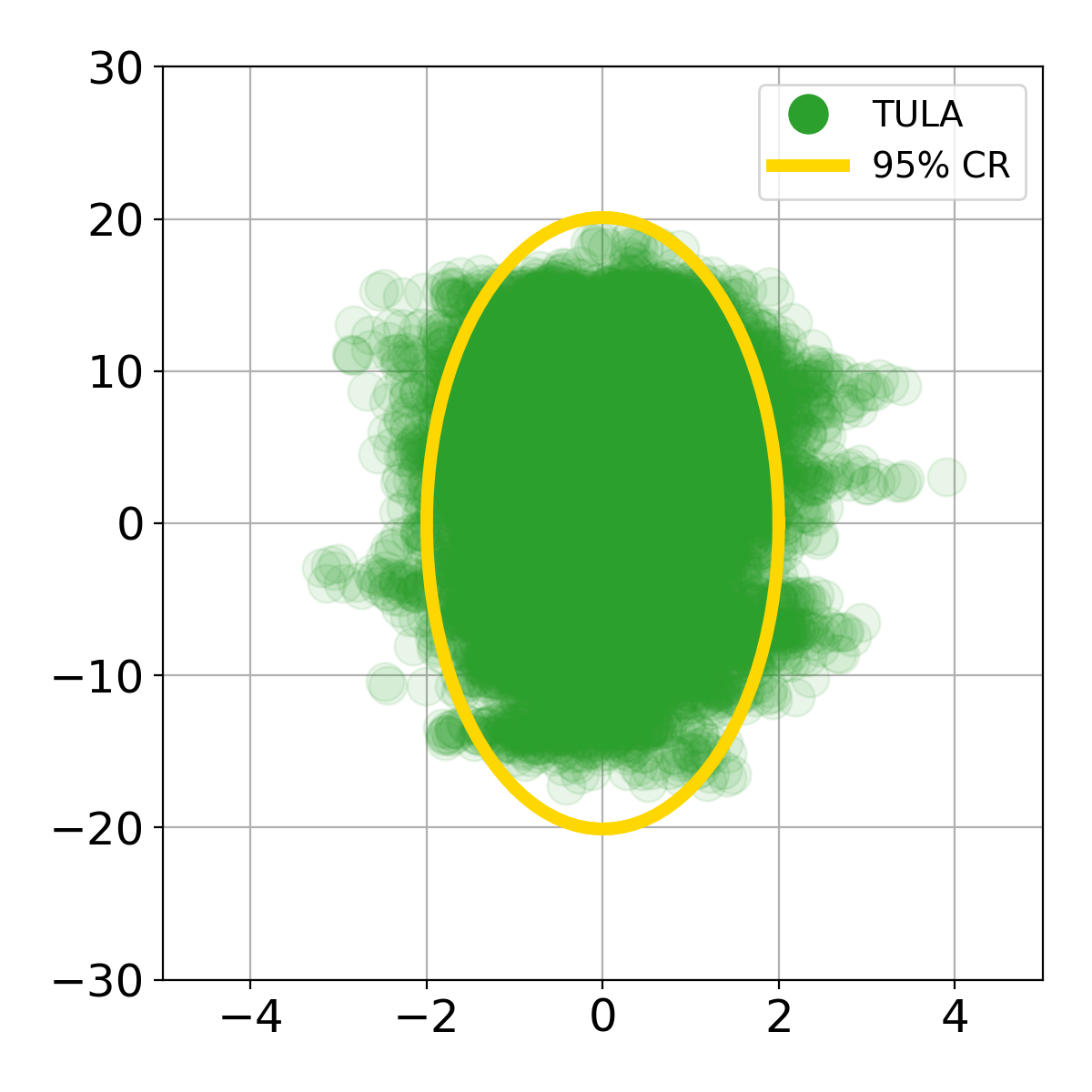}
    \end{center}
    \vspace{-.6cm}
    \caption{
    Samples from \ref{eq:NLA}, ULA, and TULA for the ill-conditioned Gaussian example of Figure~\ref{fig:gauss_mean}, with $\Sigma = \mathrm{diag}(1, 2, \dots, 100)$. We display the projection onto the first (least spread) and last (most spread) population principal components, along with the projection of a 95\% confidence region. Top: the step size for all algorithms is $h=0.7$, Bottom: the step size for all algorithms is $h=0.05$.
    }
    \label{fig:gauss_samples}
\end{figure}

\subsection{Bayesian logistic regression}
\label{subsec:lrdetails}


We give details for the two-dimensional Bayesian logistic regression example in Figure~\ref{fig:logistic}. In the Bayesian logistic regression model, covariates are drawn as $X_i \sim \mathcal{N}(0,\mathrm{diag}(10, 0.1))$, the response variables are $Y_i \sim \mathsf{Ber}(\mathrm{logit}(\theta^{\top} X_i))$, and the parameters $\theta$ have a $\mc N(0, 10I_2)$ prior. We consider using \ref{eq:NLA} to sample from the posterior distribution of $\theta$ given the observations $(X_i,Y_i)$, $i=1,\dotsc,n$, which is
$$ \pi(\theta) \propto \exp \Bigl[-\frac{1}{20} \|\theta\|^2 + \sum_{i=1}^n Y_i \theta^\top X_i - \ln(1 + e^{\theta^\top X_i}) \Bigr]\,,$$
which is strongly log-concave. While the gradient of the potential is invertible, it has no closed-form, and so in our experiments we invert it numerically by solving $\nabla V^{\star}(y) = \argmax_{x \in \R^d}{\{\langle x, y \rangle - V(x)\}}$ with Newton's method. We find that, with a warm start from the current iterate $X_t$, it suffices to run Newton's method for a small number of iterations to approximately invert the gradient. 






For the purposes of this experiment, we generate 100 samples $X_i \sim \mathcal{N}(0,\mathrm{diag}(10, 0.1))$ and $Y_i \sim \mathsf{Ber}(\mathrm{logit}(\theta^{\star\top} X_i))$, where we set $\theta^\star = (1, 1)$.

We display the result for various sampling algorithms in Figure~\ref{fig:logistic}. All algorithms are implemented with $h=0.1$ and a burn-in time of $10^4$ steps. This example shows the advantage of taking a large step-size with \ref{eq:NLA} in this ill-conditioned model, while ULA and TULA create samples that are overdispersed. In Figure~\ref{fig:logistic2}, we also show the effect of decreasing step size in this example. In this case, we see that ULA and TULA still step into low probability regions or fail to explore the underlying density well. On the other hand, \ref{eq:NLA} remains constrained in the high probability region.

\begin{figure}[ht!]
    \vspace{-0cm}
    \begin{center}
    \includegraphics[width = 0.4\textwidth]{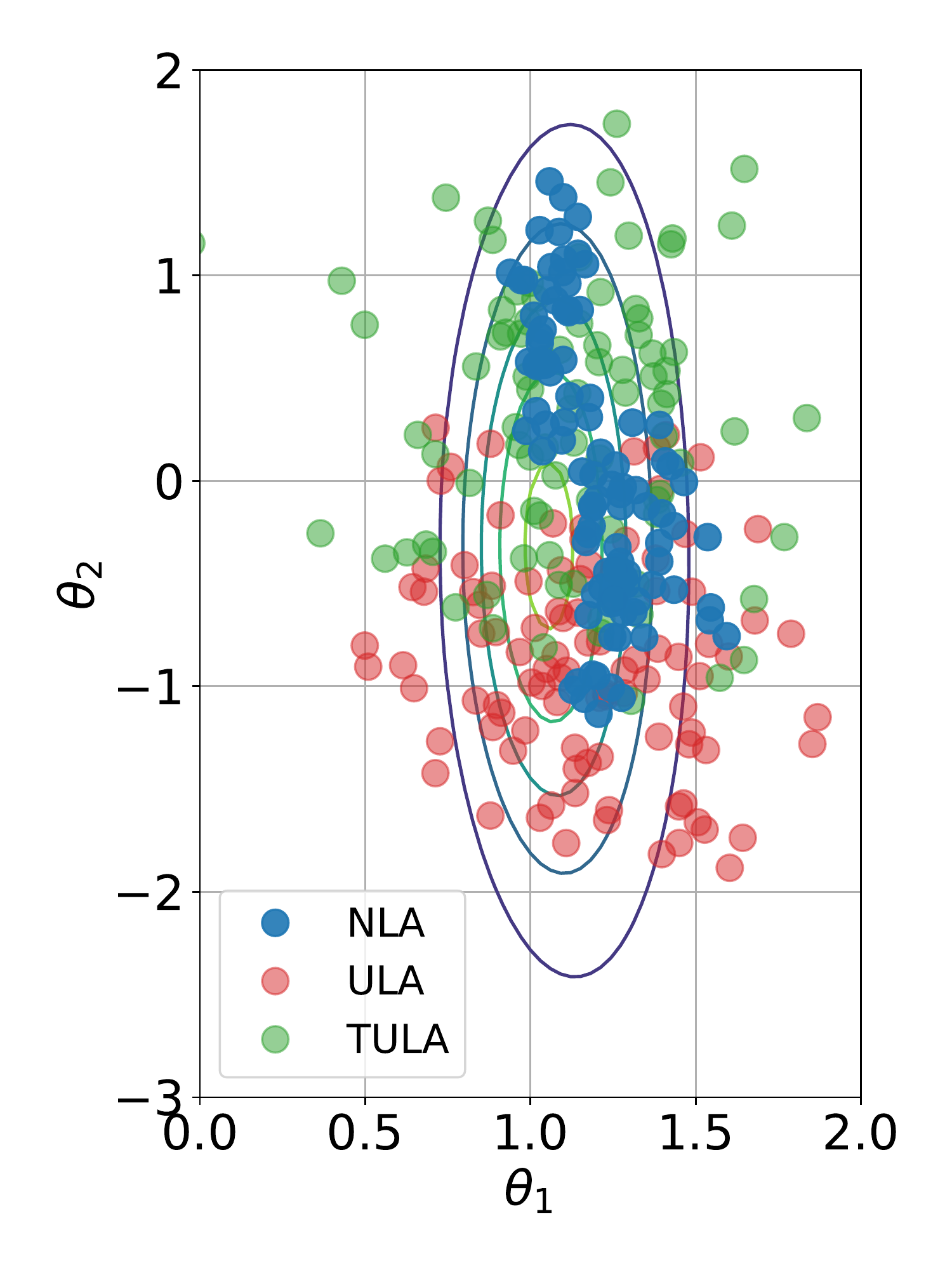}
    \includegraphics[width = 0.4\textwidth]{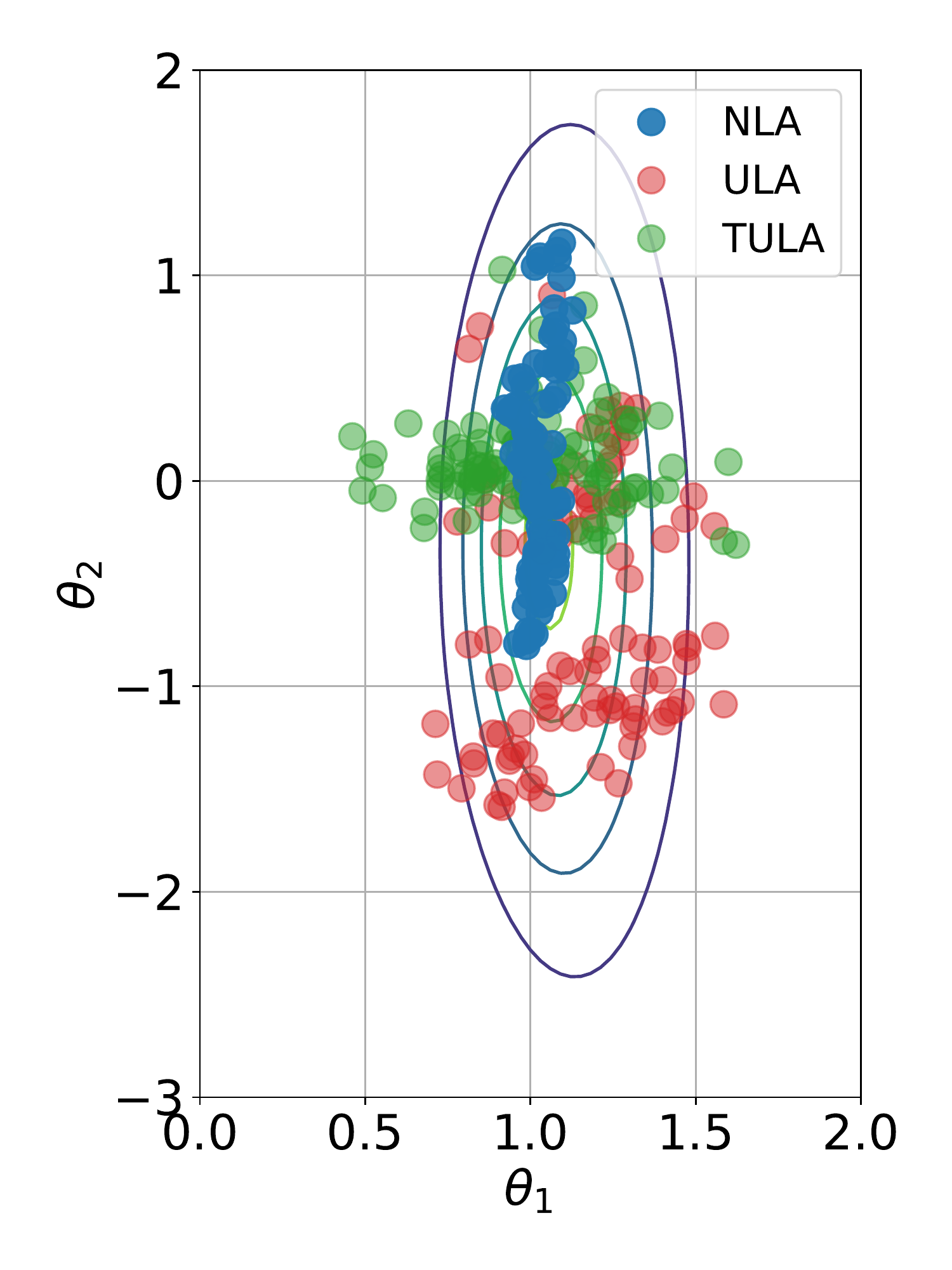}
    \end{center}
    \vspace{-.6cm}
    \caption{Samples from the posterior distribution of a Bayesian logistic regression model using one run of \ref{eq:NLA}, ULA, and TULA after a burn-in of $10^4$. Left: large step size (all algorithms use $h=0.05$); \ref{eq:NLA} remains within the high-density contours, while the ULA and TULA take steps into low-density areas. Right: small step size (all algorithms use $h=0.01$); \ref{eq:NLA} explores the underlying distribution faster than its competitors.}
    \label{fig:logistic2}
\end{figure}

\subsection{Uniform sampling on a convex body}
\label{subsec:unif_cvx}
This section contains details for the simulations in Figure~\ref{fig:ellipse}. We sample from the uniform distribution on the rectangle $[-0.01, 0.01]\times[-1,1]$ using \ref{eq:NLA}, PLA, and the Metropolis-Adjusted Langevin Algorithm (MALA)~\cite{dwivedi2019log}. PLA and MALA target the uniform distribution directly. \ref{eq:NLA} samples from an approximate distribution, given in Section~\ref{scn:unif_sampling}. The step sizes are chosen as $h=10^{-5}$ for \ref{eq:NLA} and PLA and $h=0.01$ for MALA. The step sizes for PLA and MALA are tuned to allow the algorithm to reach approximate stationarity in the fewest number of iterations. MALA can use a larger step size because it is unbiased (its stationary distribution coincides with the target distribution, due to the Metropolis-Hastings adjustment). On the other hand, samples from PLA tend to cluster around the boundary for larger step sizes, so we use a smaller step size for both PLA (and \ref{eq:NLA} for fair comparison).
\begin{figure}
    \centering
    \includegraphics[width = 0.6\textwidth]{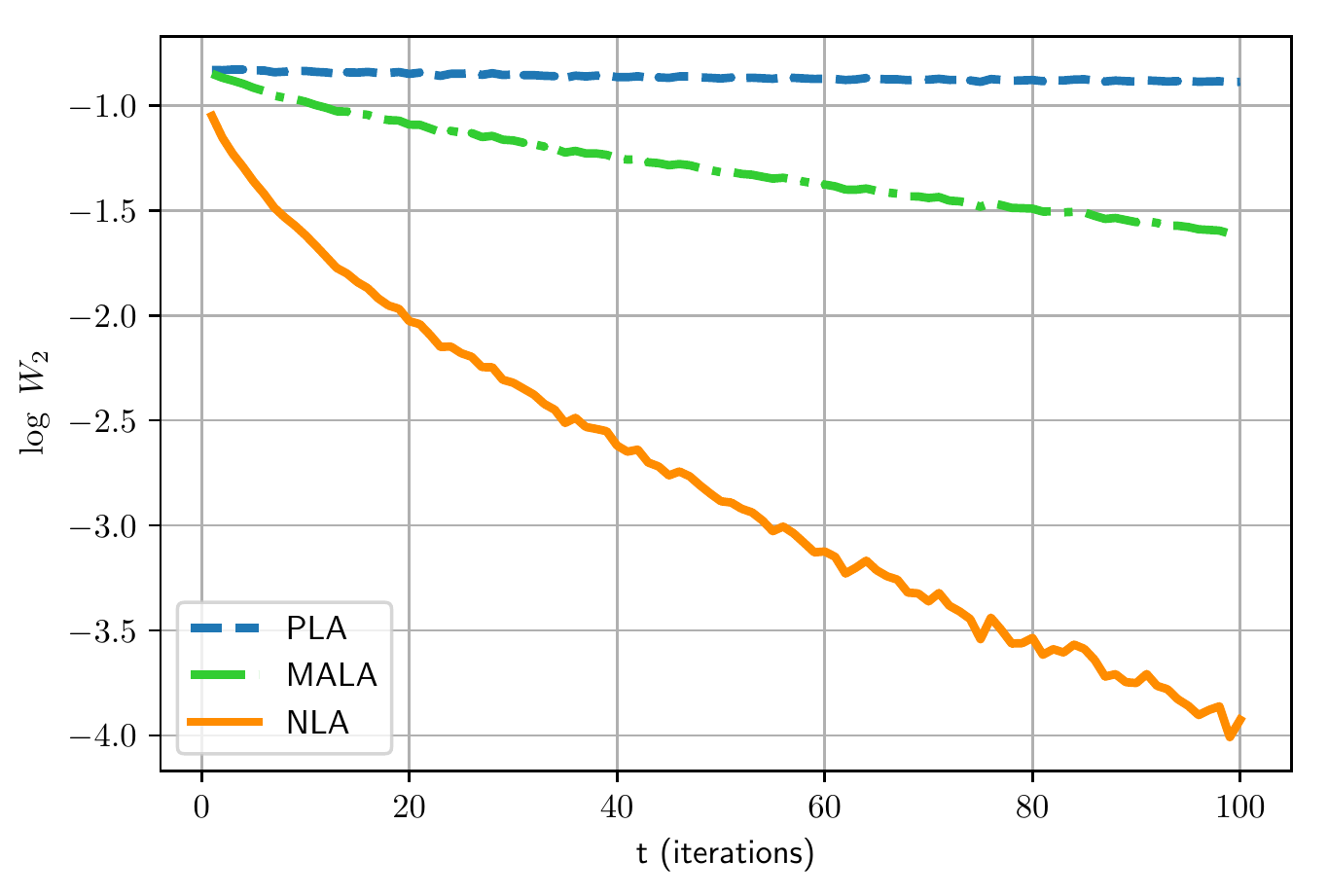}
    \caption{$W_2$ distance (on logarithmic scale) between the uniform distribution on the rectangle $[-0.01,0.01]\times[-1,1]$, and samples produced by \ref{eq:NLA}, PLA, and MALA.}
    \label{fig:sinkhorn}
\end{figure}

To evaluate the performance of the algorithms, we estimate the $2$-Wasserstein distance between the samples drawn by the algorithms and samples drawn from the uniform distribution on the rectangle; see Figure~\ref{fig:sinkhorn}. We use the Sinkhorn distance ($\eps=0.01$) as an approximation for the $2$-Wasserstein distance~\cite{cuturi2013sinkhorn, AltWeeRig17}. Specifically, we sample $1000$ points in parallel, using the three algorithms of interest. At each iteration, we also draw $1000$ points from the uniform distribution on the rectangle, and we compute the Sinkhorn distance between these points and the samples produced by the algorithms. The convergence estimates are averaged over 30 runs.

\subsection{Approximate sampling from degenerate log-concave distributions}
\label{subsec:samplestrict}

In this section, we explore further the problem of approximately sampling according to the measure $\pi(x) \propto \exp(-\|x\|)$ in $\R^2$ considered in Figure~\ref{fig:lapfull2}. To that end, we use the penalization strategy outlined in Section~\ref{subsec:newton} and sample instead from the strongly log-concave measure $\pi_\beta(x) \propto \exp(-\|x\| - \beta \|x - \mathbf{1}\|^2)$ as in Corollary~\ref{cor:degenerate_sampling}, where $\beta = 0.0005$, using discretizations of either \ref{eq:newton} or \ref{eq:mld} with a customized mirror map. Here, $ \mathbf{1}$ is the vector of all ones, which simulates the effect of not knowing the true mean.

We initialize all algorithms with a random point $X_0$ with $\norm{X_0} = 1000$. The initialization is intentionally chosen so that the gradients of the potential at initialization are extremely small. In these circumstances, we expect ULA to mix slowly.

Through this experiment, we demonstrate two empirical observations:
\begin{enumerate}
    \item Initially, the iterates of \ref{eq:NLA} converge extremely rapidly to the vicinity of the origin. This suggests that \ref{eq:NLA} can be useful for initializing other sampling algorithms in highly ill-conditioned settings.
    \item However, once the iterates of \ref{eq:NLA} are near the origin, \ref{eq:NLA} becomes unstable. Specifically, since the Hessian of the potential degenerates rapidly near $0$, the iterates of \ref{eq:NLA} occasionally make large jumps away from $0$. This is due to the fact that the Hessian of $V(x)=\|x\|+\beta \|x-\bone\|^2$
    is given by
\begin{equation}
    \label{eq:potential_norm}
      \nabla^2 V(x) = \frac{1}{\|x\|} \Bigl[I_2 - \Bigl( \frac{x}{\|x\|} \Bigr)\Bigl( \frac{x}{\|x\|} \Bigr)^\top \Bigr] + 2\beta I_2
\end{equation}
    which
    blows up to infinity around $x=0$.
    We remark that Newton's method in optimization can also exhibit unstable behavior~\cite{conn2000trust,nesterov2006cubic}, so this phenomenon is not unexpected.
    
    To rectify this behavior, we also consider the Euler discretization of \ref{eq:mld}, which we call \ref{eq:mla} (see below).
    We demonstrate that with an appropriate choice of mirror map, the iterates of \ref{eq:mla} are stable, yet still enjoy faster convergence than ULA.
\end{enumerate}

Now we proceed to the details of the experiment.
We compare four different methods for sampling from this distribution: \ref{eq:NLA}, ULA, TULA, and the mirror-Langevin Algorithm \eqref{eq:mla}
\begin{align}\label{eq:mla} \tag{$\msf{MLA}$}
    \nabla \phi(X_{k+1})
    &= \nabla \phi(X_k) - h \nabla V(X_k) + \sqrt{2h} \, {[\nabla^2 \phi(X_k)]}^{1/2} \xi_k,
\end{align}
with mirror map $\phi(x) = \|x\|^{3/2}$ and potential $V(x) = \|x\| + \beta \|x - \mathbf{1}\|^2$.
Notice that this mirror map corresponds to that used in the generalized Gaussian case of Section~\ref{sec:sim}.

In Figure~\ref{fig:lapfull}, we display the results of the first $1000$ iterations of the four algorithms. In this stage of the experiment, we observe rapid convergence of \ref{eq:NLA} towards the origin (around which the mass is concentrated), and \ref{eq:mla} also exhibits faster convergence than ULA and TULA. However, already in Figure~\ref{fig:lapfull} (Right) we observe the instability of \ref{eq:NLA} witnessed through large jumps of the iterates.

Next, in Figure~\ref{fig:lapstages}, we treat the samples from the first $1000$ iterations as burn-in, and we look at the performance of the next $1000$ samples. Here we see that the  flexible framework of the more general~\ref{eq:mld} allows us to design algorithms which can outperform \ref{eq:NLA} with superior stability in specific scenarios.

\begin{figure}[ht!]
    \vspace{-0cm}
    \begin{center}
    \includegraphics[width = 0.49\textwidth]{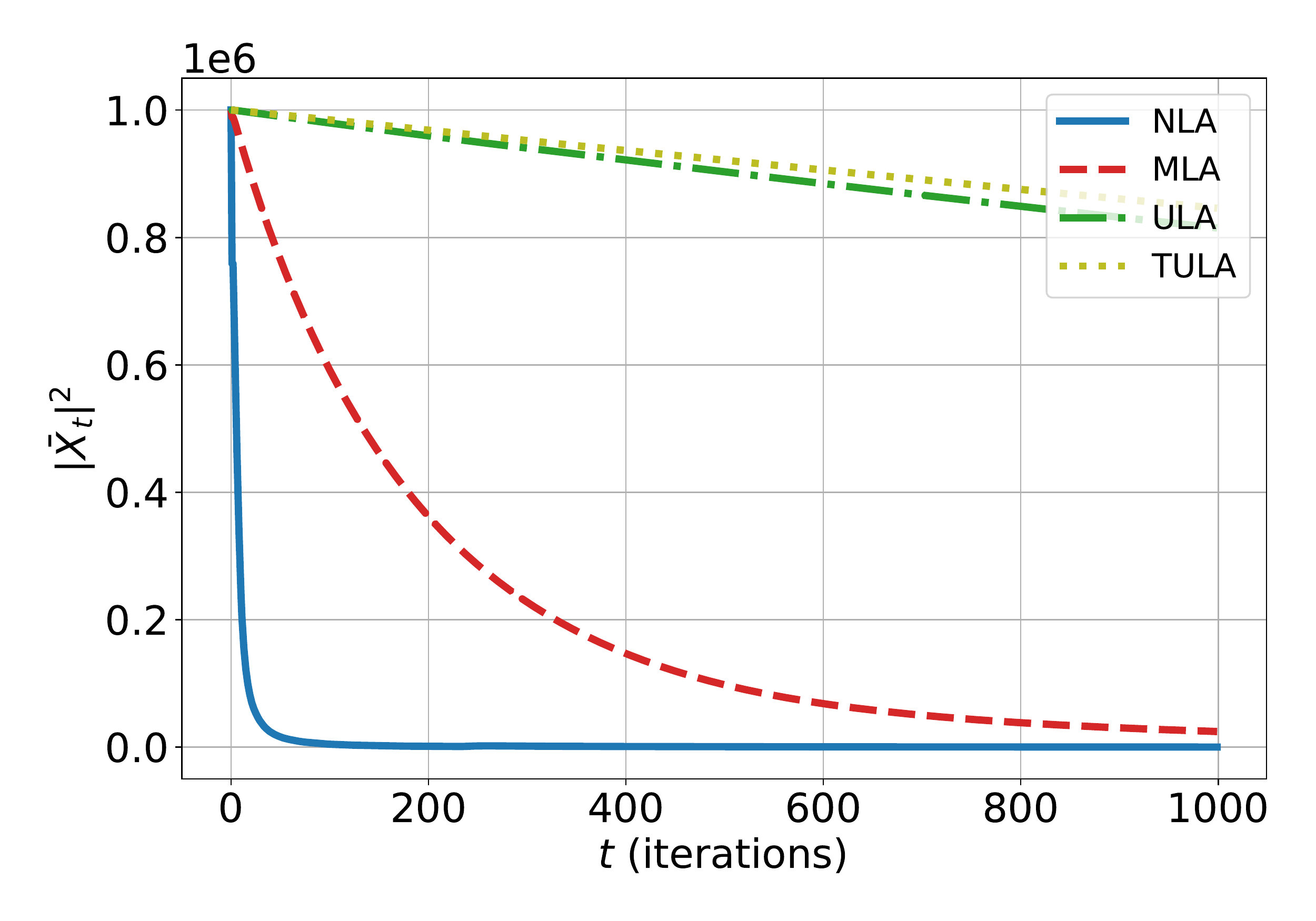}
    \includegraphics[width = 0.4\textwidth]{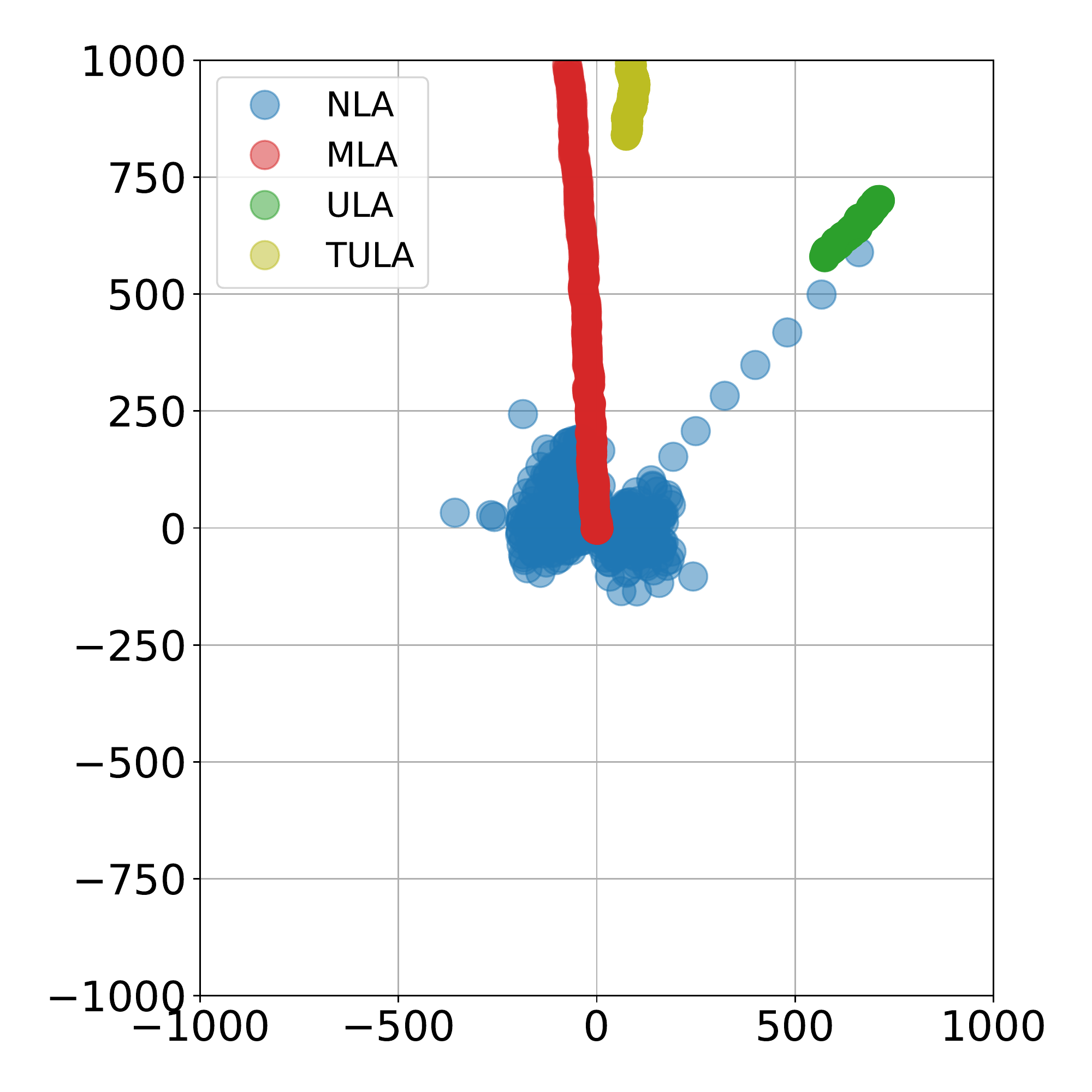}
    \end{center}
    \vspace{-.6cm}
    \caption{First stage of the experiment. Left: We plot the norm of the running mean versus the iteration number for the target measure $\pi_\beta(x) \propto \exp(-\|x\| - 0.0005 \|x - \mathbf{1}\|^2)$. Right: We display the corresponding samples.}
    \label{fig:lapfull}
\end{figure}

\begin{figure}[ht!]
    \vspace{-0cm}
    \begin{center}
    \includegraphics[width = 0.49\textwidth]{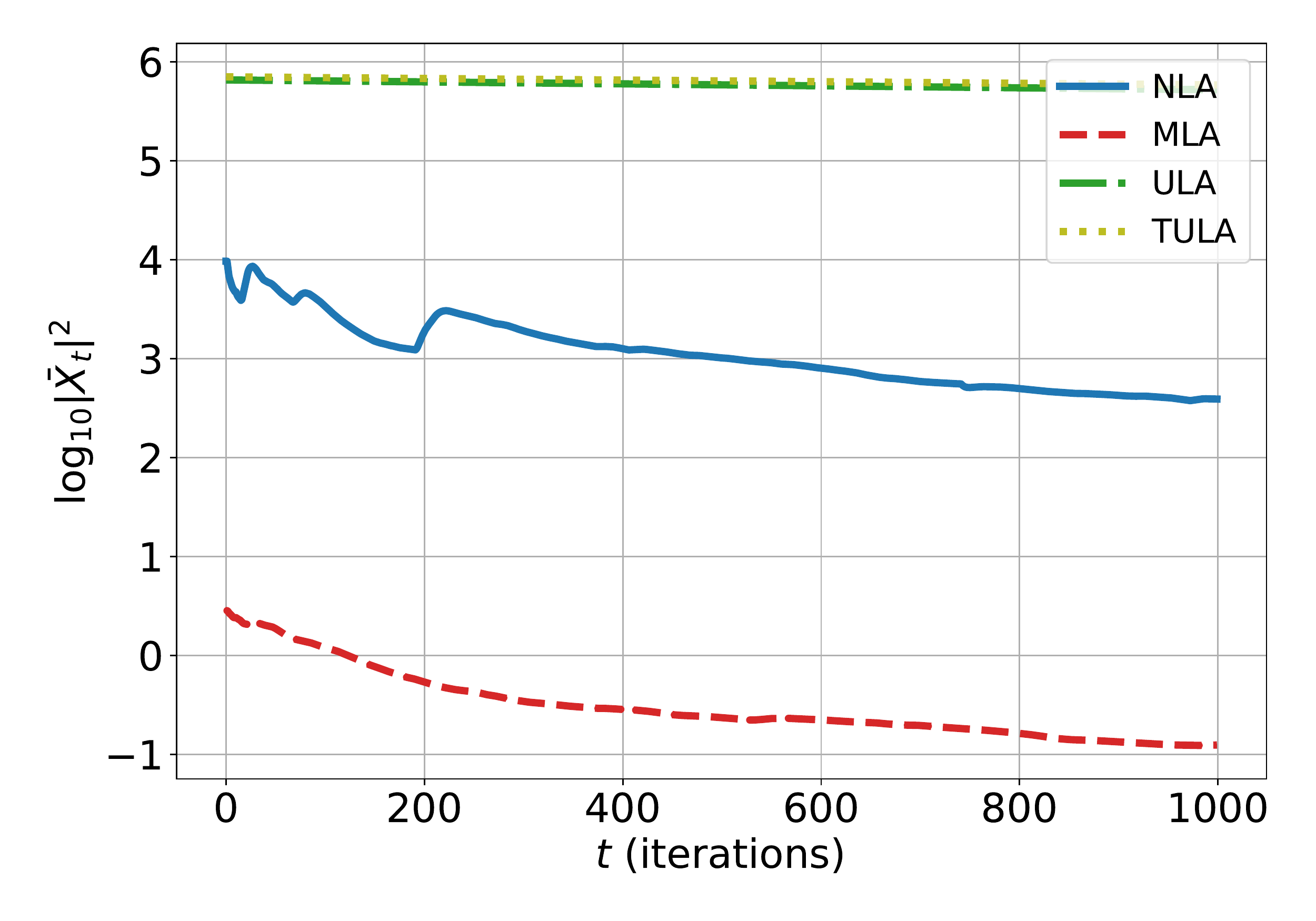}
    \includegraphics[width = 0.4\textwidth]{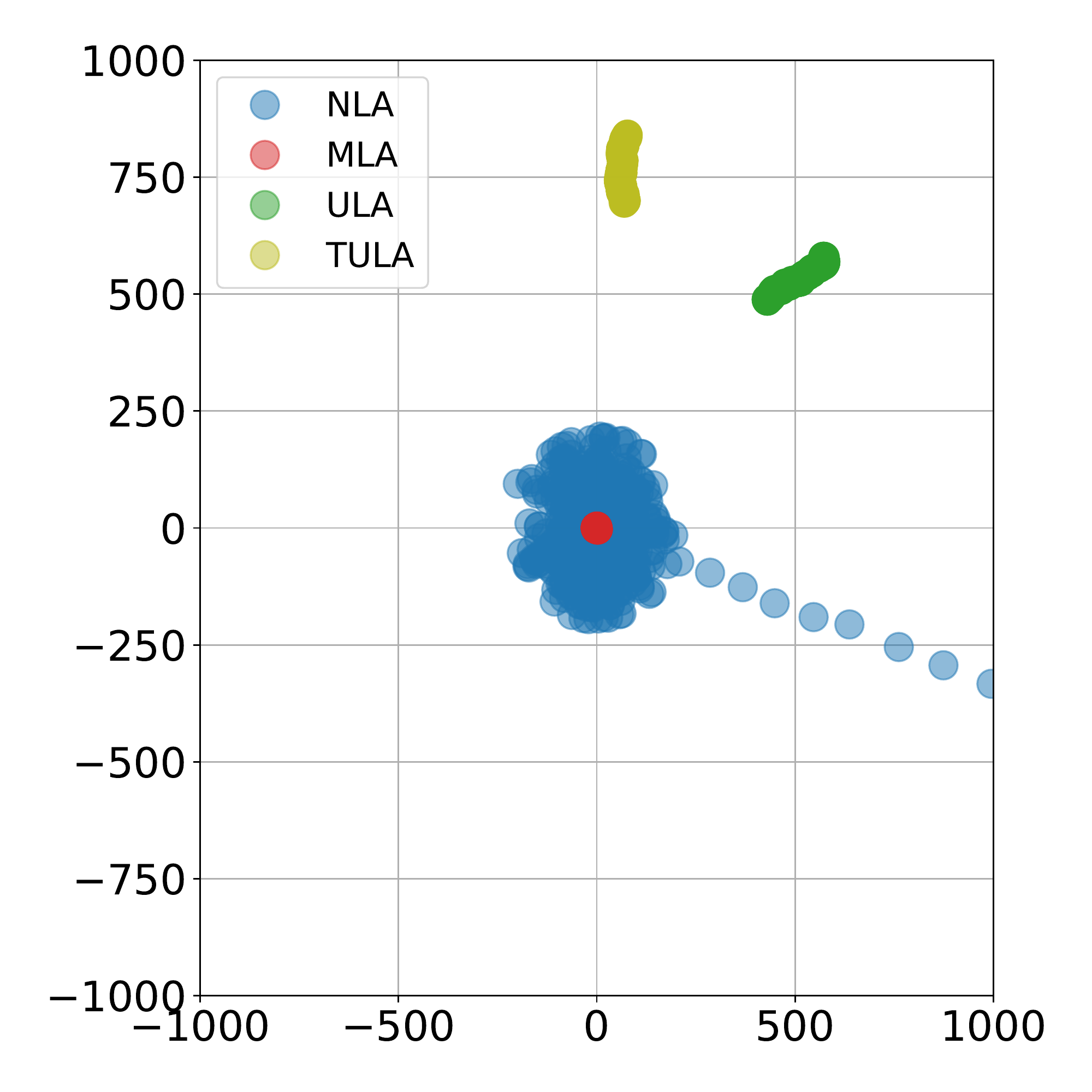}
    \end{center}
    \vspace{-.6cm}
    \caption{Second stage of the experiment. In this stage, we treat the $1000$ samples from the first stage of the experiment as burn-in and look at the performance of the next $1000$ samples. Left: We plot the logarithm of the norm of the running mean versus iteration. Right: We again display the corresponding samples.}
    \label{fig:lapstages}
\end{figure}


Recall that the Hessian of the potential $V$ is given in~\eqref{eq:potential_norm} while the potential of the mirror map $\phi$ is given by
\begin{align*}
    \nabla^2 \phi(x)
    = \frac{3}{2 \|x\|^{1/2}} \Bigl[ I_2 - \frac{3}{4} \Bigl( \frac{x}{\|x\|} \Bigr)\Bigl( \frac{x}{\|x\|} \Bigr)^\top\Bigr].
\end{align*}
From these expressions, it can be checked that Corollary~\ref{cor:potential_dominates_mirror} holds with $C \le 3/(4\sqrt{2\beta})$. On the other hand, the measure $\pi_\beta$ satisfies a Poincar\'e inequality~\eqref{eq:p} with constant $C_{\msf{P}} \le 1/(2\beta)$. Heuristically, we therefore expect the mixing time of ULA to scale as $O(\beta^{-1})$, and the mixing time of \ref{eq:mla} to scale as $O(\beta^{-1/2})$, which provides an explanation for the rates of convergence observed in Figure~\ref{fig:lapfull}. In comparison, the mixing time of \ref{eq:NLA} is scale-invariant, i.e.\ $O(1)$, as we demonstrated in Corollary~\ref{cor:conv_newton}, as witnessed by the initial rapid convergence in Figure~\ref{fig:lapfull}.

As mentioned in our open questions, this points to the intriguing possibility of developing more stable variants of \ref{eq:NLA}, which would mirror the development of such strategies for Newton's method~\cite{conn2000trust,nesterov2006cubic}.


\bibliographystyle{aomalpha}
\bibliography{ref}

\end{document}